\documentclass[reqno,11pt]{amsart}
\usepackage[T1]{fontenc}
\usepackage{geometry}
\title[Spectral asymmetry, SUSY and the equivariant Riemann-Roch defect]{Spectral asymmetry, supersymmetry and the equivariant Riemann-Roch defect}
\date{}
\author{Gayana Jayasinghe} 
\email{mgsjayasinghe@gmail.com}
\author{Alex R. Taylor} 
\email{alext3@illinois.edu}
\author{Xinran Yu}
\email{xinran4@illinois.edu}


\makeatletter
\renewcommand{\tocsection}[3]{%
  \indentlabel{\@ifnotempty{#2}{\bfseries\ignorespaces#1 #2\quad}}\bfseries#3}
  
\renewcommand{\tocsubsection}[3]{%
  \indentlabel{\@ifnotempty{#2}{\ignorespaces#1 #2\quad}}#3}
\let\tocsubsubsection\tocsubsection% Update for \subsubsection
    \renewcommand{\tocsubsubsection}[3]{%
  \indentlabel{\@ifnotempty{#2}{\ignorespaces#1 #2\quad}}#3}

\newcommand\@dotsep{4.5}
\def\@tocline#1#2#3#4#5#6#7{\relax
  \ifnum #1>\c@tocdepth % then omit
  \else
    \par \addpenalty\@secpenalty\addvspace{#2}%
    \begingroup \hyphenpenalty\@M
    \@ifempty{#4}{%
      \@tempdima\csname r@tocindent\number#1\endcsname\relax
    }{%
      \@tempdima#4\relax
    }%
    \parindent\z@ \leftskip#3\relax \advance\leftskip\@tempdima\relax
    \rightskip\@pnumwidth plus1em \parfillskip-\@pnumwidth
    #5\leavevmode\hskip-\@tempdima{#6}\nobreak
    \leaders\hbox{$\m@th\mkern \@dotsep mu\hbox{.}\mkern \@dotsep mu$}\hfill
    \nobreak
    \hbox to\@pnumwidth{\@tocpagenum{\ifnum#1=1\bfseries\fi#7}}\par% <-- \bfseries for \section page
    \nobreak
    \endgroup
  \fi}

\AtBeginDocument{%
\expandafter\renewcommand\csname r@tocindent0\endcsname{0pt}
}
\def\l@subsection{\@tocline{2}{0pt}{2.5pc}{5pc}{}}
\def\l@subsubsection{\@tocline{2}{0pt}{4.5pc}{5pc}{}}

\makeatother

\makeatletter

\renewcommand\subsection{\@secnumfont}{\bfseries}%
\renewcommand\subsection{\@startsection{subsection}{3}
  \z@{.5\linespacing\@plus.7\linespacing}{.5em}%
  {\normalfont\bfseries}}

\renewcommand\subsubsection{\@secnumfont}{\bfseries}%
\renewcommand\subsubsection{\@startsection{subsubsection}{3}
  \z@{.5\linespacing\@plus.7\linespacing}{.5em}%
  {\normalfont\bfseries}}
  
\makeatother

\usepackage{amsmath}
\usepackage{amsfonts, mathrsfs,amssymb} % mathbb, mathscr, geqslant
\usepackage{hyperref}
\usepackage{slashed} % Dirac
\usepackage{mathtools}
\usepackage{tikz-cd}
\usepackage{comment}
\allowdisplaybreaks

\usepackage{enumerate}
\DeclareMathOperator{\cl}{c\ell} %% Clifford action
\DeclareMathOperator{\Ker}{Ker} %% nullspace
\DeclareMathOperator{\Str}{Str} %% supertrace
\DeclareMathOperator{\Tr}{Tr} %% trace
\DeclareMathOperator{\sign}{sign}

\DeclareMathOperator{\Spec}{Spec} %% Spectrum

\DeclareMathOperator{\Id}{Id}

\newcommand\mf\mathfrak
\newcommand\mc\mathcal
\newcommand\mb\mathbb

\newcommand\numberthis{\addtocounter{equation}{1}\tag{\theequation}}

\numberwithin{equation}{section} %This labels equations within the section

\newtheorem{theorem}{Theorem}[section]
\newtheorem*{theorem*}{Theorem}
\newtheorem{mainthm}{Theorem}

\newtheorem{lemma}[theorem]{Lemma}

\newtheorem{remark}[theorem]{Remark}
\theoremstyle{definition}

\newtheorem{question}[theorem]{Question}

\newtheorem{example}[theorem]{Example}
\AtBeginEnvironment{example}{%
  \pushQED{\qed}%
}
\AtEndEnvironment{example}{\popQED\endexample}

\usepackage[commandnameprefix=always,  defaultcolor=orange!70!black]{changes}
\usepackage{adjustbox}
\usepackage{graphicx}
\graphicspath{ {images/} }
\usepackage{caption}
\usepackage{subcaption}
\usepackage{float}
\usepackage{comment}

\begin{document}
\pagenumbering{roman} 
\begin{abstract}
We investigate the relationship between two interpretations of equivariant Riemann-Roch defects of complex spaces with conic singularities; as (i) equivariant $\eta_{T}$ and $\xi_{T}$ invariants, and as (ii) supertraces over local cohomology groups. This leads to a novel threefold partitioning of the $L^{2}$-spinor space on the link and a corresponding splitting of $\xi_{T}$. Two partitions correspond to cohomological contributions coming from $\bar\partial$-Neumann and $\bar\partial$-Dirichlet operators on the cone, while the third partition makes no contribution to the equivariant index defect, which we show is due to supersymmetric cancellations on the cone that we call \textit{lifted supersymmetry}.

We use this to define complex equivariant $\xi_T$ and $\eta_T$ invariants, which are equivalent to the usual invariants but are easier to compute. We highlight connections to related algebraic and analytic descriptions of Riemann-Roch defects in the literature, both at the level of numbers and their categorifications, and explore connections to existing notions of supersymmetric cancellations in physics and mathematics.
\end{abstract}

\maketitle
\tableofcontents
\pagenumbering{arabic}

%------------------------------
% body
%------------------------------
\section{Introduction}

Let $X$ be a complex space with singular locus $X^{\mathrm{sing}}$. Suppose that a holomorphic vector bundle $E\to X$ is given together with a choice of sheaf of sections $\mathcal D$ on which the twisted Dolbeault operator $\overline{\partial}_E$ defines a Dolbeault complex. If a self-map $f:X\to X$ preserves the complex structure and the auxiliary data, then it induces a geometric endomorphism $T=T_f$ of the Dolbeault complex, and hence endomorphisms on its cohomology group $\mathcal H^q$.

The equivariant Riemann-Roch problem is concerned with the equivariant
Dolbeault index
\[
    \operatorname{ind}_{T_f}(\overline{\partial}_E) :=\sum_{q\ge 0}(-1)^q\Tr(T_f | \mathcal{H}^q)= \Str(T_f | \mathcal{H}^q).
\]
In the smooth compact setting, this invariant admits a purely local expression in terms of fixed-point contributions via the holomorphic Lefschetz formulas of Atiyah--Bott \cite{AtiyahBott1,AtiyahBott2} and Atiyah--Segal--Singer \cite{atiyah1965index}. When singularities are present, the same principle continues to hold, up to an additional correction term. In that case we have
\[
    \operatorname{ind}_{T_f}(\overline{\partial}_E) = \sum_{F\in \operatorname{Fix}(f)\cap X^{\mathrm{reg}}} \operatorname{RR}(F,T_f) + \operatorname{Def}(T_f,\mathcal D),
\]
where $\operatorname{RR}(F,T_f)$ denotes the local Riemann-Roch number associated to a fixed point $F$ in the regular locus, and the Riemann-Roch defect $\operatorname{Def}(T_f,\mathcal D)$ encodes the contribution of the singular locus and the chosen Dolbeault complex. The defect admits several descriptions and we study two of them in this paper, (i) an analytic formulation in terms of equivariant $\xi_T$ invariants, and (ii) an algebraic description given by supertraces over cohomology groups of neighbourhoods of the singularities. The purpose of this article is to describe a connection between these two descriptions of the equivariant Riemann-Roch defect, as well as the objects that categorify it.

In the analytic formulation, the $\eta$ and $\xi$ invariants of odd dimensional smooth manifolds~$Z$~were first introduced by Atiyah-Patodi-Singer \cite{atiyah1975spectral} to describe the index defect of Dirac-type operators $D$ on even dimensional manifolds $M$ with boundary $\partial M=Z$. These invariants are defined in terms of spectral data of an induced Dirac-type operator $A$ on the boundary. Briefly, the boundary contribution coming from the nonzero spectrum of $A$ is the \emph{equivariant $\eta$ function}
\[
    \eta_T(s,A):=\sum_{\lambda \in \mathrm{Spec}(A)\setminus\{0\}} \sign(\lambda) |\lambda|^{-s} \Tr (T|E_{\lambda}), 
\]
where $\Tr(T|E_{\lambda})$ denotes the trace over the (finite dimensional) eigenspace $E_{\lambda}$ (see Remark \ref{Remark_different_renormalizations} for alternative renormalizations in the physics literature). The function is holomorphic when we have $\mathrm{Re}(s) > -2$, and admits a meromorphic continuation to $\mathbb{C}$ with a regular value at $s = 0$, cf. \cite{weiping1990note}. Thus one may consider the \emph{equivariant $\eta$ invariant}, defined as the value $\eta_{T}(0,A)$. The \emph{equivariant $\xi$ function} is then
\[
    \xi_{T}(s) := \frac{1}{2}(\eta_{T}(s,A) + h_{T})
\]
where $h_{T} = \Str(T|_{L^{2}\Ker A})$ is the boundary contribution coming from the nullspace of $A$. The value $\xi_{T}(0)$ is precisely the index defect in the equivariant Lefschetz fixed point formula in \cite{weiping1990note} for spaces with isolated conic singularities, and in the equivariant APS index theorem \cite{donnelly1978eta}.  It specializes to the Riemann-Roch defect for the equivariant Dolbeault index.

On the other hand, Baum-Fulton-MacPherson-Quart investigated (equivariant) Riemann-Roch formulae \cite{baum1979lefschetz,baumfultonmacKtheoryRiemannRoch,baumfultonmacpheresonRiemannRoch} on certain singular spaces where the equivariant Riemann-Roch defect (in singular cohomology) was presented as a regularized trace over the local ring on locally complete intersections, described in the expository article \cite{Baumformula81} by Baum. This interpretation was extended as a supertrace on local cohomology groups on stratified pseudomanifolds with wedge metrics (iterated conic metrics) and wedge almost complex structures for twisted Spin-c Dirac operators, and twisted Dolbeault-Dirac operators when the wedge almost-complex structure is integrable in \cite{jayasinghe2023l2,jayasinghe2024holomorphic}. These cohomology groups correspond to harmonic forms of corresponding Laplace-type operators with appropriate boundary conditions.

This leads to a natural question: \textbf{how are the spectral data corresponding to these two descriptions linked?} The main results of the present article answer this question: for a twisted Dolbeault-Dirac operator on a cone $C(Z)$ equipped with a complex structure $J$, the spectrum of the induced Dirac operator on the link $Z$ admits a special threefold partition. Two parts of this partition correspond to local cohomology groups of Dolbeault complexes on the cone $C(Z)$, while the third part is non-trivial but does not contribute to the equivariant $\xi$ invariant, due to cancellations that correspond to supersymmetry on the lifts of the eigensections on the link to those on the cone (see Remark \ref{Remark_why_call_it_lifted_SUSY}). We refer to this phenomenon as \textit{\textbf{lifted supersymmmetry}}.

For clarity of exposition, we focus in this paper on cones over smooth links, although the results extend more generally to stratified pseudomanifolds equipped with wedge metrics (see for instance, \cite{Albin_2017_index,jayasinghe2023l2}). Let $C(Z) = C_x(Z) = \big([0,\infty)_x \times Z\big)\big/ \big(\{0\} \times Z \big)$ denote a cone over a compact manifold $Z$, equipped with a wedge metric
\[
    g = dx^2 + x^2 g_Z,
\] 
a compatible wedge complex structure $J$, and a Hermitian vector bundle $E$.
The twisted Dolbeault-Dirac operator on $C(Z)$ admits a factorization
\[
    D = c\ell(dx) \circ \left( \nabla^E_{\partial_x} + A_{1}+A_{2} \right),
\]
where $A_{2}$ is the \emph{transversal part} of the Dirac operator (where the derivatives are with respect to vector fields transverse to $R=xJ(\partial_{x})$, which in the K\"ahler setting is the Reeb vector field). The even/odd degree complex spinor bundles $F^{\pm}$ admit direct sum decompositions $F^{\pm}=G^{\pm} \oplus \beta \wedge G^{\mp}$, where $G^{\pm}$ are the spinor bundles for the transversal Dirac operator, i.e., sections valued in the~$(0,q)$ projection of the bundle $\Lambda^*(TZ/T\mathcal{F}_{R}) \otimes E$ where $\mathcal{F}_R$ is the foliation given by integrating the vector field $R$. We refer to Section \ref{section_wedge_geo} for more details.

We define the \textbf{restriction operator} which evaluates smooth sections of the Dolbeault complex on $C(Z)$ at $x=1$:
\begin{align*}
    &\mathcal{R}: \Omega^{0,q}(C(Z);E) \to \Gamma(Z;E|_{Z}) \\
    & \mathcal{R}(\psi(x,z))=\psi(1,z),
\end{align*}
which also extends to sections which are smooth near $x=1$.

\medskip
Given a Dirac-type operator $A$ on $Z$, we say that it satisfies: 
\begin{enumerate}
    \item the \textbf{Witt condition} if $\Spec(A) \cap \{0\} = \varnothing$;
    \item the \textbf{spectral Witt condition} if $\Spec(A) \cap (-\tfrac{1}{2}, \tfrac{1}{2}) = \varnothing$.   
\end{enumerate}

Our first main result, Theorem \ref{intro: mainthm}, describes a spectral decomposition for the operator $A_{1}$~restricted to $Z$. The result also makes an explicit correspondence between the local cohomology groups and the eigensections of $A_{1}$ on the link. In particular, there is a four-part description of the eigensections of $A_{1}$ lying in the transverse harmonic subspaces $\mathcal{E}^{\pm} = \ker A_{2} \cap \Gamma(Z;F|^{\pm}_Z)$.

Given a Dolbeault operator on a cone $C(Z)$ as above, we pick self-adjoint domains, corresponding to ideal boundary conditions at~$x=0$, denoted by $W$, and boundary conditions at~$x=1$, denoted by $B=N,D$ ($\overline{\partial}$-Neumann conditions and the adjoint boundary condition which we call the $\overline{\partial}$-Dirichlet condition respectively). This defines a complex $\mathcal{P}_{W,B}$ and we define the cohomology groups $\mathcal{H}_B^q = \mathcal{H}^q(\mathcal{P}_{W,B}):=\ker D$ where $D$ is equipped with the corresponding domain. If the spectral Witt condition holds, all choices of $W$ are equivalent.

With this setup, we define the \textbf{complex equivariant $\xi$ function}
\begin{equation}
    \widetilde{\xi}_{W,T}(s):=\frac{1}{2} \Big(\Str \left(T\exp (-s\sqrt{\Delta_Z})|_{\mathcal{H}(\mathcal{P}_{W,N}(X))}\right) + \Str\left(T^*\exp (-s\sqrt{\Delta_Z})|_{\mathcal{H}(\mathcal{P}^*_{W,N}(X))}\right) \Big)
\end{equation}
which is regular at $s=0$, and we drop the subscript $W$ from the notation in the case of the spectral Witt condition; see Subsection \ref{subsection_complex_equiv_xi} for more details. We remark that this description of the Riemann-Roch defect as a supertrace is natural, given that the Riemann-Roch number itself is a supertrace over cohomology. 

\begin{mainthm}[Spectral decomposition and vanishing of the third sector]\label{intro: mainthm}
    Assume the spectral Witt condition for the induced Dirac operator on $Z$. 
    Then:
    
    \begin{enumerate}
        \item There exists an orthonormal basis $\{\psi_j\}\subset \mathcal H_N^q\oplus \mathcal H_D^q$ such that $\mathcal R(\psi_j)$ is an eigensection of the induced Dirac operator $A$ on $Z$.  With respect to the splitting $F^\pm|_Z = G^\pm \oplus \beta\wedge G^\mp$, $\bar\partial$-Neumann boundary conditions yield boundary values in $G^\pm$, while $\bar\partial$-Dirichlet boundary conditions corresponds to a $\beta$-shift of $G^\pm$:
        \[
            \mathcal{R}\bigl(\mathcal H_N^{\mathrm{even/odd}}\bigr) \subset \mathcal E \cap G^{\pm}, 
            \qquad
            \mathcal{R}\bigl(\mathcal H_D^{\mathrm{even/odd}}\bigr) \subset \beta\wedge\bigl(\mathcal E \cap G^{\mp}\bigr).
        \]
    
        \item The even-degree spinors on $Z$ admit an orthogonal decomposition
        \[
            L^2\Omega^{0,q}(Z;F^+|_Z)  = \mathcal H_1^q \oplus \mathcal H_2^q \oplus \mathcal H_3^q,
        \]
        where $\mathcal H_1^q=\mathcal R(\mathcal H_N^q)$, $\mathcal H_2^q=\mathcal R(\mathcal H_D^q)$, and $\mathcal H_3^q=\mathcal E^\perp$. This in turn yields a threefold splitting of the $\xi_T$ function:
        \begin{equation}\label{threefold-partition}
            \xi_{T}(s) = \xi_{T,1}(s) + \xi_{T,2}(s) + \xi_{T,3}(s),
        \end{equation}
        cf. Equation \eqref{threefold-partition 2} for more details. 
    
        \item For geometric endomorphisms $T=f_g^*$ with no fixed points on the link $Z=\{x=1\}$, the equivariant and complex equivariant $\xi$ invariants coincide:
        \[
            \widetilde{\xi}_{T}(0)=\xi_{T}(0).
        \]
        \textbf{In particular, $\mathcal H_3^q$ does not contribute to $\xi_T$.}
    \end{enumerate}
\end{mainthm}

The first two parts of Theorem \ref{intro: mainthm} are a restatement of Theorem \ref{mainthm} while part three is Theorem~\ref{Theorem_simpler_invariant_computation}.

We interpret Theorem \ref{intro: mainthm} in terms of supersymmetry, or rather a lifted supersymmetry at the level of the cone; (see \cite{jayasinghe2023l2}, Lemma 3.13 in particular, and \cite{jayasinghe2025SUSY} for descriptions of supersymmetry on cones). Theorem \ref{intro: mainthm} shows that the partition $\mathcal{H}^q_3$ does not contribute to the equivariant index defect because of renormalized supersymmetric cancellations leading to $\xi_{T,3}(0)=0$. However, the fact that $\xi_{T,3}(s)$ does not vanish for general $s$ shows that there is no corresponding cancellation happening at the level of the alternating traces over the eigensections of the operator on the link.

In this sense, the cancellation is made visible only after lifting to the cone. Our thesis is that the renormalizations used in the classical definition of the $\eta$ invariant are not well-suited to making the underlying supersymmetric cancellations transparent. This motivates our definition of the new equivariant $\xi$ function $\widetilde{\xi}_{W,T}(s)$, meant to capture the supersymmetric cancellations more clearly.

When the Witt condition fails, the Dolbeault-Dirac operator $D$ need not be essentially self-adjoint, and self-adjoint extensions are characterized by Lagrangian subspaces of the Gelfand-Robbin quotient. We study the case of rolled-up extensions $\mathcal{D}_{\mathfrak{L}_W}$ determined by \textbf{cohomological pre-domains $W$}, which induces compatible domains for a self-adjoint extension of $D$. We now formulate the Riemann-Roch defect for such rolled-up extensions of the Dolbeault-Dirac operator in the following theorem, the proof of which is in Subsection \ref{Subsection_equiv_ind_def}.

\begin{mainthm} 
    \label{intro: Theorem_general_domains_xi_invariants}
    Consider the setting of Theorem \ref{intro: mainthm} but without the spectral Witt condition. For any rolled-up extension corresponding to a choice of cohomological pre-domain $W$, the $\widetilde{\xi}_T$ invariants satisfy
    \begin{equation}
       \widetilde{\xi}_{T,W} =\widetilde{\xi}_{T,\min}+\frac{1}{2}\Big(\Str (T|_{W^q})- \Str (T^*|_{(W^*)^q}) \Big).
    \end{equation}  
\end{mainthm}

This is similar in spirit to other formulations of the index defect corresponding to domains for Dirac-type operators, for instance in \cite{chou1985dirac,chou1989criteria} (see Remark \ref{Remark_on_Chou_Zhang_work}).

The setting of metric cones with complex structures demonstrates the various phenomena investigated in this paper, where equivariant $\eta$ invariants and index defects have been studied in detail. However the results of this article can be easily extended to cones over stratified pseudomanifolds with wedge metrics. It seems likely that the results may extend \textit{mutatis mutandis} to other singular metrics on such spaces, as well as to spaces with non-integrable almost complex structures that respect the geometry of the singularity. We explore this and other relevant themes in Section \ref{section_conclude}.

The article is organized as follows. Section \ref{section_wedge_geo} of this article introduces the main objects of study, the Dolbeault complexes and their domains on cones with complex structures. In Section~\ref{section_main_results} we define the new complex equivariant $\widetilde{\eta}_T$ and $\widetilde{\xi}_T$ invariants, and prove our main results. While many examples have been interspersed to illustrate the need for various definitions and assumptions, Section \ref{section:examples} goes over key examples. In Section \ref{section:examples} we exhibit the computational efficacy of the $\widetilde{\eta}_T$ and~$\widetilde{\xi}_T$ invariants in comparison with the classical invariants. In Section \ref{section_conclude} we expand on the context of our results and make concluding remarks, in particular showing relations to other approaches to the Riemann-Roch defect in mathematics and relationships to physical interpretations. We explain conjectural extensions which we hope to address in upcoming work.

\medskip
\noindent \textbf{Acknowledgements.} 
The authors thank Pierre Albin for several discussions, his helpful comments and encouragement. GJ and XY thank Hadrian Quan and Manousos Maridakis for discussions on related topics and ideas. GJ thanks the University of Peradeniya for support in various forms while pursuing this research.
\newpage\section{Wedge Geometry and Dirac Operators} \label{section_wedge_geo}

In this section we introduce the geometric and functional analytic setting of this article.

\subsection{Wedge geometry and Dirac operators}

We consider the metric cone $\widehat{X} = C(Z) = \big([0,\infty)_x \times Z\big)\big/\big(\{0\}\times Z\big)$, over a smooth $n$-dimensional manifold $Z$, where $x$ denotes the radial variable and $\{0\}\times Z$ is collapsed to a point.  We equip $\widehat{X}$ with the \textbf{wedge metric}
\[
    g = dx^{2} + x^{2} g_{Z},
\]
where $g_{Z}$ is a fixed Riemannian metric on $Z$. The space $\widehat{X}$ is singular at the cone point and its resolution is the manifold with boundary $X_\infty = [0,\infty)_x \times Z$ whose metric is degenerate at~$x=0$. We will also work on the truncated model
\[
    X = [0,1)_x \times Z,
\]
equipped with the metric obtained by restricting $g$.

For the wedge metric, the natural differential operators to study are those belong to the \textbf{wedge tangent bundle} $\prescript{w}{}{TX}$, which is locally spanned by 
\[
    \Big\{e_0 = \partial_x, \quad e_1 = \frac{1}{x}\partial_{z_1},\quad \cdots \quad e_n = \frac{1}{x}\partial_{z_n} \Big\},
\]
where $(z^i)$ is a coordinate chart on $Z$. The dual bundle of $\prescript{w}{}{TX}$ is the \textbf{wedge cotangent bundle}~$\prescript{w}{}{T^*X}$, whose local coframe is
\[
    \{ \theta^0 = dx, \quad \theta^1 = x dz^{1}, \quad \cdots, \quad \theta^n = x dz^n \}, 
\]
and hence may be written as $\prescript{w}{}{T^{*}X} = \operatorname{span}\{dx\} \oplus x\, T^{*}Z$. For a detailed discussion of the wedge tangent and cotangent bundles, as well as the wedge complex structures and wedge Dirac operators that we develop next, we refer the reader to \cite{Albin_2016_Index, jayasinghe2025SUSY}.

We call a closed, anti-symmetric, non-degenerate bi-linear wedge 2-form $\omega$ a \textbf{wedge symplectic form}. A section $J$ of the endomorphism bundle of $\prescript{w}{}{T^*X}$ satisfying $J^2 = -\Id$ is called a \textbf{wedge almost complex structure}. If, in addition, $J$ satisfies $\nabla J = 0$, then $J$ is said to be \emph{integrable}, and it defines a \emph{wedge complex structure} on $X$.

A \textbf{wedge Clifford module} consists of a complex vector bundle $E \to X$, a Hermitian bundle metric $g_E$, a connection $\nabla^E$ compatible with $g_E$, and a \emph{Clifford multiplication} given by a bundle homomorphism
\[
    c\ell: \mathbb{C}l_w(X) = \mathbb{C} \otimes Cl(\prescript{w}{}{T^*X}; g)\to \mathrm{End}(E)
\]
such that 
\[
    c\ell(\theta)c\ell(\eta)+c\ell(\eta)c\ell(\theta) = -2g(\theta,\eta)\,\Id_E,
\]
and for every $W \in \prescript{w}{}{TX}$ and $\theta \in \prescript{w}{}{T^*X}$ the following compatibility conditions hold:
\[
    g_E(c\ell(\theta)\cdot, \cdot ) = -g_E(\cdot, c\ell(\theta) \cdot ) \quad \text{and} \quad [\nabla^E_W, c\ell(\theta)] = c\ell(\nabla^E_W \theta).
\]
A \textbf{Dirac-type operator} $D$ associated to a Hermitian bundle $E$ is defined by the composition
\[C_c^\infty(\mathring{X}; E) \xlongrightarrow{\nabla^E} C_c^\infty(\mathring{X}; T^*X \otimes E) \xlongrightarrow{c\ell} C_c^\infty(\mathring{X}; E).\] 
Writing $D$ in an orthonormal frame adapted to the conic geometry near the boundary, one obtains the factorization
\begin{equation} \label{eq:Dirac-factorization-1}
    D = \sum_{i=0}^n c\ell(\theta^i) \nabla^E_{e_i}  = c\ell(dx)\left(\nabla^E_{\partial_x} + A\right),
\end{equation}
The operator $A$ is the tangential (twisted) Dirac operator on the link $Z$ that governs the asymptotic behavior of $D$ near the cone tip. Let $J$ denote the complex structure on the cone $X$, defined by $J(\partial_x) = \frac{1}{x}R, J(dx) = x\alpha$. When $C(Z)$ is Kähler, the link $Z$ carries a Sasakian structure, with Reeb vector field $R$ and contact 1-form $\alpha$.

We decompose the vector field $\partial_x$ into its $(0,1)$- and $(1,0)$-components as 
\begin{equation}
    \beta^\# = \frac{1}{2} \left( \partial_x + \sqrt{-1} J \partial_x \right), \quad \overline{\beta}^\# = \frac{1}{2} \left( \partial_x - \sqrt{-1} J \partial_x \right),
\end{equation}
with dual 1-forms given by $\beta = dx - \sqrt{-1} x \alpha$ and  $\overline{\beta} = dx + \sqrt{-1} x \alpha$.

We adopt the Clifford action conventions of \cite[Equation (3.3)]{wu1998equivariant}, under which the Clifford multiplications by $\overline{\beta}^\#$ and $\beta^\#$ are given by
\begin{equation} \label{Clifford mutiplication}
    c\ell(\overline{\beta}^\#) = \sqrt{2} \beta \wedge, \quad c\ell(\beta^\#) = -\sqrt{2} \iota_{\beta^\#}.
\end{equation}
We define the Clifford element
\begin{equation} 
    \nu := \sqrt{2} \left( \beta \wedge - \iota_{\beta^\#} \right).
\end{equation}
Note that $\nu^2 = -2$. The the Dirac operator \eqref{eq:Dirac-factorization-1} can be written as 
\begin{equation} \label{eq:Dirac-factorization-2}
    D = \nu \circ \left( \nabla^E_{\partial_x} -\frac{\nu}{2x} \circ (D_1 + D_2) \right) = \nu \circ \left( \nabla^E_{\partial_x} + A \right),
\end{equation}
where
\[
    D_1 = -  \sqrt{2}\left( \beta \wedge  +\iota_{\beta^\#} \right) \nabla^E_{\sqrt{-1}R},
\]
the operator $D_2$ is the transversal Dirac operator on the link, and $A$ is given by
\begin{equation} \label{equation_B_expansion}
    A = -\frac{\nu}{2x} \circ (D_1 + D_2)
    = -\frac{1}{x}\left( \beta \wedge - \iota_{\beta^\#} \right) \circ \left( 
    \left( \beta \wedge + \iota_{\beta^\#} \right) \nabla^E_{\sqrt{-1}R} 
    + \frac{1}{\sqrt{2}} D_2 \right).
\end{equation}
Moreover, we will define the operators
\begin{equation} \label{equation_operators_123}
    A_{1} = -\frac{\nu}{2x}D_{1} \mbox{ \, and \, } A_{2} = -\frac{\nu}{2x}D_{2}
\end{equation}
so that $A = A_{1}+A_{2}$.

In the Dolbeault case, the even/odd degree complex spinor bundles $F^{\pm}$ admit direct sum decompositions
\begin{equation} \label{equation_Dolbeault_decomposition}
    F^{\pm}=G^{\pm} \oplus \beta \wedge G^{\mp}
\end{equation}
where $G^{\pm}$ are the spinor bundles associated with the transversal Dirac operator. Concretely, their sections take values in the $(0,q)$ projection of the bundle $\Lambda^*(TZ/T\mathcal{F}_{R})$, with $\mathcal{F}_R$ denoting the foliation given by integrating the vector field $R$.

\begin{remark}[anti-commutation] \label{Remark_sign_switch_clifford_eigs_1} 
    The Clifford element $\nu$ anti-commutes with the transversal Dirac operator $D_Z$ since $\nu$ acts by right Clifford multiplication. In particular, $\nu$ anti-commutes with $D_2$ since $\beta$ is orthogonal to the directions appearing in the local Clifford description of~$D_2$~and that $\nu$ anti-commutes with $\beta \wedge + \iota_{\beta^\#}$. 
    
    As a consequence, $\nu$ intertwines the summands of Equation \eqref{equation_Dolbeault_decomposition} and interchanges the~$F^{\pm}$~bundles. 
\end{remark}

\subsection{Twisted Dolbeault complexes and domains}

We now introduce the framework of twisted Dolbeault complexes on conic spaces and describe how the Gelfand–Robbin quotient characterizes the admissible domain extensions of the associated Dirac operator.

Let $X = C(Z)$ denote an infinite cone, equipped with the wedge metric $g = dx^2 + x^2 g_Z$ as in previous subsection. Consider the case where there is a complex structure $J$ on the interior of the cone which extends to a wedge complex structure. Fix a Hermitian line bundle $E \to X$~endowed with a unitary connection $\nabla^E$. We denote the twisted Dolbeault operator $P = \overline{\partial}_E$ in the usual manner on the smooth twisted anti-holomorphic forms on $X$. This yields a \textit{pre-Hilbert complex}~$(H_\bullet,P_\bullet)$, which, together with a domain $\mathcal{D}(P_{\bullet})$ for $P_\bullet$ that is closed under the graph norm, determines a Hilbert complex
\[
    \mathcal{P} = (H_\bullet, \mathcal{D}(P_{\bullet}), P_\bullet), 
    \qquad
    0 \rightarrow \mathcal{D}(P_0) \xlongrightarrow{P_0} \mathcal{D}(P_1) \xlongrightarrow{P_1} \cdots \xlongrightarrow{P_{n-1}} \mathcal{D}(P_n) \rightarrow 0,
\]
where $H_\bullet = L^2\Omega^{0,\bullet}(X;E)$. Every Hilbert complex $\mathcal{P}$ admits an \emph{adjoint Hilbert complex} 
\begin{multline*}
    \mathcal{P}^* = (H_{n-\bullet}, \mathcal{D}^*(P_{n-\bullet})=\mathcal{D}(P^*_{n-\bullet}), P^*_{n-\bullet}), \\
    0 \rightarrow \mathcal{D}(P_{n-1}^*) \xlongrightarrow{P_{n-1}^*} \mathcal{D}(P_{n-2}^*) \xlongrightarrow{P_{n-2}^*} \cdots \xlongrightarrow{P_1^*} \mathcal{D}(P_0^*) \rightarrow 0,
\end{multline*}
where each $P_k^*$ denotes the Hilbert space adjoint of $P_k$. The \textbf{(twisted) Dirac-type operator} associated with the Dolbeault complex is $D = P + P^*$.

We start with picking self-adjoint extensions on infinite cones. Since the incompleteness of the metric is near the singularities (as opposed to the radial infinity), the choices of self-adjoint extensions corresponds to choices of local data near the singularities. We recall the standard maximal and minimal domains of differential operators $Q$ are given by
\begin{align}
    \mathcal{D}_{\max}(Q) &= \{ u \in L^2\Omega^{0,\bullet}(X;E) \mid Qu \in L^2\Omega^{0,\bullet}(X;E) \text{ in the distributional sense}\},\\
    \mathcal{D}_{\min}(Q) &= \text{graph closure of } C_c^\infty\Omega^{0,\bullet}(\mathring{X};E).
\end{align}
It is easy to observe that $\mathcal{D}_{\min}(Q) \subseteq \mathcal{D}_{\max}(Q)$. The \textbf{Gelfand–Robbin quotient} is defined to be
\begin{equation}
    \mathcal{GR}(Q) = \mathcal{D}_{\max}(Q) \big/ \mathcal{D}_{\min}(Q).
\end{equation}

In the case of an infinite cone over a smooth manifold $Z$ (i.e. the metric closure of $(0,\infty)_x \times Z$ with respect to the wedge metric $g=dx^2+x^2g_Z$), $\mathcal{GR}(D)$ for Dirac-type operators $D$ are finite-dimensional. These have been studied, for instance in \cite{mooers1999heat,Bei_2012_L2atiyahbottlefschetz,Albin_2017_index,jayasinghe2024holomorphic}, where the spaces can be understood in terms of the small eigenvalue eigensections of the induced operator $A$ on $Z$. We will provide a concrete example (see Example \ref{Example_Complex_cone_6pi}). The choices of domains correspond to what are known as \textit{ideal boundary conditions}, similar to boundary conditions on spaces with boundary.

For any \textbf{symmetric operator} $Q$, the space $\mathcal{GR}(Q)$ carries a natural symplectic form induced by the Green–Stokes identity.  
For $[u],[v] \in \mathcal{GR}(Q)$, define
\begin{equation} \label{eqn: green-stokes}
    \omega([u],[v]) = \langle Q u, v\rangle_{L^2} - \langle u, Q^* v\rangle_{L^2} 
    = \langle i\sigma_1(Q)(-dx) u|_Z,\, v|_Z \rangle_{L^2(Z)},
\end{equation}
where $\sigma_1(Q)(-dx)$ denotes the principal symbol of $Q$ evaluated at the dual form $-dx$ of the outward pointing normal vector. The anti-symmetry is apparent since $Q=Q^*$. A choice of closed Lagrangian subspace $\mathfrak{L} \subseteq \mathcal{GR}(Q)$ specifies a self-adjoint extension of $Q$, with domain
\[
    \mathcal{D}_\mathfrak{L}(Q) = \{ u \in \mathcal{D}_{\max}(Q) \mid [u] \in \mathfrak{L} \}.
\]
The requirement that $\mathfrak{L} \subseteq \mathcal{GR}(Q)$ be Lagrangian is precisely the condition that the Green–Stokes pairing vanish identically on $\mathfrak{L}$, i.e. $\omega([u],[v]) = 0$, for all $[u],[v] \in \mathfrak{L}$.

While twisted Dolbeault operators $P=\overline{\partial}_E$ are not symmetric, the choices of domains still correspond to subspaces of $\mathcal{GR}(P)$. We define
\[
    \mathcal{H}(P) = (\ker P_{\max} \cap \ker P^*_{\min}) \cap \mathcal{D}_{\min}(P)^{\perp_{L^2}}.
\]
Note that $\mathcal{H}(P)$ can be identified with the cohomology group of the \textit{Gelfand–Robbin complex}
\[
    0 \rightarrow \mathcal{GR}(P_0) \xlongrightarrow{\widetilde{P}_0} \mathcal{GR}(P_1) \xlongrightarrow{\widetilde{P}_1} \cdots \xlongrightarrow{\widetilde{P}_{n-1}} \mathcal{GR}(P_n) \rightarrow 0,
\]
where $\widetilde{P}_k: [u] \mapsto [P_k u]$. For this reason, we refer to $\mathcal{H}(P)$ as the \textit{cohomology quotient}. We choose a subspace ${W} \subseteq \mathcal{H}(P)$ and call such a choice a \textit{cohomological pre-domain}. We define a topological subspace
\[
    \mc{D}_{W}'(P) := W \oplus \mc{D}_{\min}(P) \subseteq L^{2}\Omega^{0,\bullet}(X;E).
\]
Since the adjoint of a domain is closed, we see that $\mathcal{D}_W(P):=(\mathcal{D}_{W}'(P))^{**}$ is a closed domain for $P$. The adjoint domain is given by $\mathcal{D}_{W^{*}}(P^*):=(\mathcal{D}_{W}'(P))^{*}=(\mathcal{D}_W(P))^{*}$, and we note that $\mathcal{D}_{W^*}(P^*)$ can be constructed by an analogous procedure as in the case of $P$ by choosing a subspace $W^* \subseteq \mathcal{H}(P^*)$. Identifying $\mathcal H(P)$ with a cohomology group of the Gelfand-Robbin complex of $P^*$, each $v \in \mathcal H(P)$ determines a unique class  $[v] \in \mathcal{GR}(P^*)$. Then we can write~$W^*$~as  
\[
    W^{*} := \{v \in \mathcal{H}(P^*) \mid [v] \in \mathcal{GR}(P^{*}), \omega([u],[v]) = 0 \mbox{ for all } u \in W \},
\]
where $\omega$ is the Green-Stokes pairing defined in Equation \eqref{eqn: green-stokes}.

\subsection{Rolled-up extension and generalized boundary conditions}

We now define a self-adjoint extension of the twisted Dolbeault-Dirac operator $D = P + P^*$ associated with a cohomological pre-domain $W \subset \mathcal H(P)$. We set
\[
    \mathcal D_{\mathfrak L_W}(D) := \mathcal D_W(P) \cap \mathcal D_{W^*}(P^*).
\]
By construction, $\mathcal D_{\mathfrak L_W}(D)$ defines a self-adjoint domain for $D$. We refer to $\mathcal D_{\mathfrak L_W}(D)$ as the \textbf{rolled-up extension} of $D$ corresponding to the domain $\mathcal D_W(P)$.

Any self-adjoint extension of $D$ corresponds to a choice of Lagrangian subspace 
\[
    \mathfrak{L} \subseteq \mathcal{GR}(D) = \mathcal{D}_{\max}(D) / \mathcal{D}_{\min}(D).
\]
A chosen domain of $P$ canonically determines a compatible domain of $D$ and we work with such rolled-up extensions in this article. Later, in Theorem \ref{intro: Theorem_general_domains_xi_invariants} we will explain how Theorem~\ref{Theorem_simpler_invariant_computation} extends to the case of rolled-up extensions; essentially, the extra sections in $W$ contribute additional supertraces.

The following example shows that not every self-adjoint extension of $D$ is rolled-up.
\begin{example}
Consider $C(S^3)$, where the volume of the sphere is much smaller than the volume of the standard round sphere in $\mathbb{C}^2$. Then the Gelfand-Robbin of the Dolbeault operator includes the section 
\begin{equation}
    s=z_1^{-1}z_2^{-1}+d\overline{z_1}\wedge d\overline{z_2}
\end{equation}
which is an even degree spinor. However, $s$ does not fall in the Gelfand-Robbin quotients of either $P$ or $P^*$. So one can pick some Lagrangian of $\mathcal{GR}(D)$ which contains $s$ and it is clear that the corresponding extension of the Dolbeault-Dirac operator is not a rolled-up extension of the Dolbeault operator.
\end{example}

When studying operators on the finite cones corresponding to the metric closure of $(0,1)_x \times Z$, in addition to ideal boundary conditions at $x=0$, we need to pick boundary conditions at $x=1$~in order to have a choice of self-adjoint extension for the operators.

For the Dolbeault complex, there are two canonical choices of boundary conditions at $x=1$, by choosing the minimal and maximal extensions for $P$ for sections localized near the boundary. These are also called the \textbf{generalized Neumann and Dirichlet boundary conditions} for~$P$, given by
\begin{equation}\label{dirichlet-neumann-bdy-conditions}
    \sigma(P^{*})(dx)u|_{x=1} = 0, \hspace{4mm} \sigma(P)(dx)u|_{x=1} = 0
\end{equation}
respectively. We refer the reader to \cite{jayasinghe2023l2} for more details. For the case of $P=\overline{\partial}$, the generalized Neumann condition leads to the $\overline{\partial}$-Neumann condition for the corresponding Laplace-type operator $D^2$.

Thus given a closed extension $\mathcal{D}_W(P)$ on the infinite cone, we study two domains on the truncated cone, denoted by $\mathcal{D}_{W,B}(P)$ where $B=N,D$ denotes the generalized Neumann or Dirichlet condition. We denote the corresponding Hilbert complex by $\mathcal{P}_{W,B}(P)$ (as $\mathcal{P}_{W,B}$ when the operator $P$ is clear by context) and the cohomology groups by
\begin{equation}\label{local-cohomology-groups}
    \mathcal{H}(\mathcal{P}_{W,B}):= \ker D \subseteq \mathcal{D}_{W,B}(P) \cap (\mathcal{D}_{W,B}(P))^*.
\end{equation}

We now illustrate the preceding constructions using an explicit model of the Dolbeault complex on a conical singularity, showing how the inclusion of certain sections in a self-adjoint domain forces the exclusion of others through non-vanishing boundary pairings.

\begin{example}[Dolbeault complex on a complex cone]\label{Example_Complex_cone_6pi}
Let $X = C_x(S^1_{6\pi})$ denote the cone over the circle of length $6\pi$, equipped with the conic metric $g = dx^2 + x^2 d\theta^2$. We endow $X$ with a complex structure $J$ defined by
\[
    J(\partial_x) = \frac{1}{x} \partial_\theta, \quad J(\partial_\theta) = -x \partial_x,
\]
which makes $(X,g,J)$ into a K\"ahler cone. We introduce the \textit{singular complex coordinate} $z := x e^{i\theta}$ with $\theta \in S^1_{6\pi}$ and analyze the Dolbeault complex on $X$.

One can check that the harmonic representatives in the maximal domain but not in the minimal domain are spanned by
\[
    \tau_1 = c_1{z}^{-1/3}, \quad \tau_2 = c_2{z}^{-2/3}, \quad s_1 = c_2\overline{z}^{-2/3} \overline{dz}, \quad s_2 = c_1\overline{z}^{-1/3} \overline{dz}.
\]
where $c_1,c_2$ are the unique constants such that the sections are of unit norm on the cone (with respect to the volume form $-2i(dz \wedge \overline{dz})=dx \wedge xd\theta$). The section $\tau_1 = c_1z^{-1/3}$ is holomorphic in the interior and $L^2$-integrable with respect to the volume form $\mathrm{dvol}_g = dx \wedge x d\theta$. Similarly,~$s_1~=~c_1 \overline{z}^{-2/3} \overline{dz}$ belongs to the null space of $\overline{\partial}^*$ and is $L^2$-integrable.

Applying the Green–Stokes formula to $P = \overline{\partial}^*$ yields the boundary pairing
\begin{equation}
    [s_m,\tau_n]:=\langle \overline{\partial}^* s_m, \tau_n \rangle_E - \langle s_m, \overline{\partial} \tau_n \rangle_E = \int_{0}^{6\pi} g_E( i\sigma_1(\overline{\partial}^*)(dx) s_m, \tau_n) \, x d\theta,
\end{equation}
where $E = \Lambda^a T^*X^{0,1} $ is the wedge anti-holomorphic form bundle and the pairing $g_E$ denotes the Hermitian structure on $(0,1)$-forms. The principal symbol acts as 
\[
    i\sigma_1(\overline{\partial}^*)(dx)(\overline{dz}) =  \iota_{dx^\#|_{T^{0,1}}}(\overline{dz}) = \iota_{e^{-i\theta} \overline{\partial_z}} (\overline{dz}) = e^{-i\theta}.
\]
since $dz=e^{i\theta}(dx+xid\theta)$.

Since 
\begin{equation}
    g_E(\overline{z}^{-2/3} e^{-i\theta}, {z}^{-1/3}) = \overline{z}^{-2/3} e^{-i\theta} \overline{z}^{-1/3}= x^{-2/3} e^{+2i\theta/3}e^{-i\theta} x^{-1/3} e^{+i\theta/3} = x^{-1},
\end{equation}
the boundary pairing $[s_1,\tau_1]$ is given by
\begin{equation} \label{equation_inner_product_0}
    \frac{1}{c_1c_2}[s_1,\tau_1]=\int_0^{6\pi} g_E(\overline{z}^{-2/3} e^{-i\theta}, {z}^{-1/3}) \, x d\theta = \int_0^{6\pi} 1\, d\theta =6\pi \neq 0.
\end{equation}
Hence the boundary pairing does not vanish, indicating that $\tau_1$ and $s_1$ cannot simultaneously belong to a self-adjoint domain of $\overline{\partial} + \overline{\partial}^*$. One can similarly check that
\begin{equation}
    \frac{1}{c_1c_2}[s_1,\tau_1]= \frac{1}{c_2c_1}[s_2,\tau_2]=6\pi, \quad [s_1,\tau_2]=[s_2,\tau_1]=0.
\end{equation}
\end{example}

The choices of domain for the Dolbeault operators at the level of cohomology corresponds to a choice of vector subspace of $\mathcal{H}(\overline{\partial})=\text{span} \{ \tau_1, \tau_2\}$, which then gives a unique rolled-up extension for $D$ that corresponds to a Lagrangian subspace of the boundary pairing for $D$.

For instance consider the pre-domain at the level of cohomology 
\[
    W_{a,b}=\text{span}\{a\tau_1+b\tau_2\} \subseteq \mathcal{H}(P),
\]
for some $a,b \in \mathbb{R}^+$ such that $
(a^2+b^2)=1$. Then if the corresponding Lagrangian subspaces corresponding to the rolled-up extensions for $D$ are $\mathcal{L}_{a,b}$, the intersection of which with the kernel of the Dirac operator $D$ corresponds to $\text{span}\{a\tau_1+b\tau_2, bs_1-as_2\}$. The corresponding pre-domain at the level of cohomology for the $P^*$ operator is the span of $bs_1-as_2$.

In the example above, the last choice of domain we selected for the Dolbeault–Dirac operator does not arise from an algebraic domain determined by a weight, as studied in \cite{jayasinghe2024holomorphic} (or as an algebraic domain as studied in \cite{Vertman2015Heatalgebraic}). Later in Example \ref{example_Lefschetz_traces_0}, we will compute equivariant~$\xi_T$~invariants for the above domain.
\section{Main results} \label{section_main_results}

\subsection{Review of $\eta$ and $\xi$ invariants}

We review  definitions of equivariant $\eta$ invariants for geometric endomorphisms $T$ on $Z$, that correspond to geometric endomorphisms $T$ on $C_x(Z)$. For the case of a self-map $f_g$ (for some $g \in G$) induced by an isometric action of a Lie group $G$ that fixes only the cone point (which in particular would preserve the radial distance $x$ on the cone), where the action preserves the complex structure, then the geometric endomorphism on the Dolbeault complex is given by~$f^*_g$~(the pullback by the self-map) which commutes with the operator on the given domains. We refer to \cite{AtiyahBott1,Bei_2012_L2atiyahbottlefschetz,jayasinghe2023l2} for more details on more general geometric endomorphisms of Dolbeault complexes.

\textit{In this article we shall study geometric endomorphisms $T$ that are generated by actions of Lie groups that act by isometries and preserve the complex structure, while fixing only the cone point.} Let there be a geometric endomorphism $T=f^*_g$ on the Dolbeault complex that acts on the radial function $x$ by the identity morphism, and thus also acts on the eigensections of the Dirac operator on the link $Z$. Then there is an adjoint endomorphism $T^{*}$ acts on the dual complex, given by the pullback of sections by $f^*_g$.

We will consider in this article two functions $\xi_{T}(s)$ and $\widetilde{\xi}_{T}(s)$ associated with the operator~$A$~from Equation \eqref{eq:Dirac-factorization-2} and the geometric endomorphism $T$. The first function $\xi_{T}(s)$ is the standard equivariant $\xi$ function associated with the equivariant index defect at the singularity in \cite{weiping1990note}, motivated by the study of the boundary contribution in the equivariant APS index theorem in \cite{donnelly1978eta}. The function $\xi_{T}(s)$ is the sum of two boundary contributions: one coming from the nonzero spectrum of $A$ and one from the nullspace.

The boundary contribution for the equivariant index coming from the nonzero spectrum of~$A$~is the \emph{equivariant $\eta$ function}
\begin{equation}\label{eqn:equiv-eta-fn}
    \eta_T(s,A):=\sum_{\lambda \neq 0} \sign(\lambda) |\lambda|^{-s} \Tr (T|_{E_{\lambda}}), 
\end{equation}
where the sum is over the nonzero eigenvalues $\lambda$ of the operator $A$ acting on the restriction to~$Z$~of the even degree sections of the Dolbeault complex (those with values in $F^+$), and where~$\Tr(T|_{E_{\lambda}})$~denotes the trace over the (finite dimensional) eigenspace for the eigenvalue $\lambda$. The equivariant $\eta$ function is holomorphic on $\Re(s) > -2$, and admits a meromorphic continuation to $\mathbb{C}$ with a regular value at $s = 0$, cf. \cite{weiping1990note}. Thus one may consider the \textbf{equivariant~$\eta$~invariant}, defined as the value $\eta_{T}(0,A)$.

We denote by $\pi:X \rightarrow Z$ the projection to a given link at $x=1$. Given a section $s$ of $F|_{Z}$ ($F=\Lambda^{p,\bullet}(X;E)$), we denote by $\pi^*(s)$ the section of $F$ that is covariantly constant with respect to the $\partial_x$ vectors and satisfies $\pi (\pi^*(s))=s$. Given a section $u$ of $F|_{Z}$, we say that it extends as an $L^2$-bounded section of the truncated cone if $\pi^*(u)$ is $L^2$-bounded in an open neighbourhood of the point $x=0$. We denote by $L^{2}\Ker A$ the elements $u$ of the null space of $A$ that extend to $L^2$-bounded sections $\pi^*u$ on the truncated cone (on both the even and odd degree sections of~$F|_{Z}$). One defines
\begin{equation}
    h_T:=\Str\left(T|_{{L^2 \Ker A}}\right) = \Tr(T|_{L^{2}\Ker A^{+}}) - \Tr(T|_{L^{2}\Ker A^{-}})
\end{equation}
and this is the boundary contribution coming from the nullspace of $A$. The standard definition of the equivariant $\xi$ function is 
\begin{equation}\label{equiv-xi-function}
    \xi_{T}(s):=\frac{1}{2} \Big(\eta_T(s,A)+h_{T} \Big).
\end{equation}
The value $\xi_{T}(0)$ is precisely the equivariant index defect in the equivariant Lefschetz fixed point formula in \cite{weiping1990note}, and in the equivariant APS index theorem \cite{donnelly1978eta}.

\subsection{Complex equivariant $\xi$ function}
\label{subsection_complex_equiv_xi}

For a Hilbert complex $\mathcal{P} = (H_\bullet, \mathcal{D}(P_{\bullet}), P_\bullet)$ on a cone $X = C(Z)$, the Dirac operator $D = P+P^{*}$ admits a factorization
\[
    D = \cl(dx)(\nabla_{\partial_{x}}^{E} + \frac{1}{x}A_{Z})
\]
where $A_{Z} = xA$ is the tangential operator on the link $Z$. The associated tangential Laplace-type operator on the link is defined as $\Delta_{Z} = A_{Z}^{2}$. In particular, $\Delta_{Z}$ is a nonnegative self-adjoint elliptic operator on $Z$, and since $Z$ is compact, the spectrum of $\Delta_{Z}$ is discrete:
\[
    \Spec(\Delta_{Z}) = \{\alpha_{i}\}_{i \geq 0}, \hspace{3mm} 0 \leq \alpha_{0} \leq \alpha_{1} \leq \cdots,
\]
with finite-dimensional eigenspaces. By the spectral theorem we may therefore define
\[
    |A_{Z}| = (A_{Z}^{2})^{1/2} = \sqrt{\Delta_{Z}},
\]
as well as the operators
\[
    e^{-s |A_{z}|}, \hspace{3mm} |A_{Z}|^{-s}, \hspace{3mm} s \in \mathbb{C}
\]
by the functional calculus. The Laplacian $\Delta_{Z}$ acts on sections over the link $Z$, and separation of variables identifies the cone harmonic sections with modes indexed by tangential eigensections of $\Delta_{Z}$. Thus, after choosing an orthonormal basis $\{h_{q,i}\}_{i \geq 0}$ of the local cohomology space~$\mathcal{H}^{q}(\mathcal{P}_{W,B}(X))$ adapted to the tangential spectral decomposition, we can define the induced action of $e^{-s|A_{z}|}$ on local cohomology by
\[
    e^{-s|A_{Z}|}h_{q,i} = e^{-s\sqrt{\alpha_{i}}}h_{q,i}.
\]
With this understood, for a given cohomological predomain $W \subset \mathcal{H}(P)$, we define the \textbf{complex equivariant $\xi$ function} to be
\begin{equation}\label{new-equiv-xi-function}
    \widetilde{\xi}_{T,W}(s):=\frac{1}{2} \Big(\Str \left(T\exp (-s\sqrt{\Delta_Z})|_{\mathcal{H}(\mathcal{P}_{W,N}(X))}\right) + \Str\left(T^*\exp (-s\sqrt{\Delta_Z})|_{\mathcal{H}(\mathcal{P}^*_{W,N}(X))}\right) \Big)
\end{equation}
where $\mathcal{P}^*_{W,N}$ is the adjoint complex of $\mathcal{P}_{W,N}$. The function $\widetilde{\xi}_{T,W}(s)$ is holomorphic on $\Re(s)~>~0$, where the sum over tangential modes converges; moreover it will follow from Theorem \ref{Theorem_simpler_invariant_computation} that~$\widetilde{\xi}_{T,W}(s)$ admits a meromorphic expansion in a neighborhood of $s=0$ with a regular value at $s=0$. The \textbf{complex equivariant $\xi$ invariant} is then defined as the value $\widetilde{\xi}_{T,W}(0)$ at $s=0$. We will drop the subscript $W$ from the notation in the case of the spectral Witt condition.

For geometric endomorphisms $T$ that correspond to group actions on $X$ that only fix the point at $x=0$, the Lefschetz-Riemann-Roch results in \cite{jayasinghe2023l2}, (cf. \cite{jayasinghe2024holomorphic}) show that the equivariant index defect at $x=0$ is given by the evaluation at $s=0$ of either one of
\begin{equation}
    \Str \left(T\exp (-s\sqrt{\Delta_Z})|_{\mathcal{H}(\mathcal{P}_{N}(X))}\right), \quad \Str\left( T^*\exp (-s\sqrt{\Delta_Z})|_{\mathcal{H}(\mathcal{P}_{D}(X))}\right)
\end{equation}
which are equal (at $s=0$) by dualities of the complexes $\mathcal{P}_N(X)$ and $\mathcal{P}_D(X)$. It is easy to see that this extends to an equality at $s=0$ of 
\begin{equation}
    \Str \left(T|{\Delta_Z}|^{-s/2}|_{\mathcal{H}(\mathcal{P}_N(X))}\right), \quad \Str \left(T^*|{\Delta_Z}|^{-s/2}|_{\mathcal{H}(\mathcal{P}_D(X))}\right)
\end{equation}
by changing the weights in the supertraces from $\exp (-s\sqrt{\Delta_Z})$ to $|{\Delta_Z}|^{-s/2}$, using the same techniques as those in  \cite[\S 5.2.5]{jayasinghe2023l2}. The regularization used in our definition of $\widetilde{\xi}_{T,W}(s)$ is analogous to the regularization in \cite[(3.3)]{cvetic2025extradimensionaletainvariantsanomalytheories}, where the classical $\eta$ invariant is given as the limit
\[
\lim_{\epsilon \to 0^{+}} \sum_{k} \exp(-\epsilon|\lambda_{k}|)\,\mathrm{sign}(\lambda_{k})
\]
where $\lambda_{k}$ are the nonzero eigenvalues of the induced Dirac operator on the boundary.

\begin{remark}
\label{Remark_different_renormalizations}
    The definition of $h_{T}$ on manifolds with boundary, and even on spaces with conic singularities, differs slightly depending on choices of domains and conventions in various articles. The definitions of $\eta$ and $\xi$ invariants reflect those differences as well, but the index defect for these different definitions will match in examples.
    
    We also note that sometimes different regularizations are used for the $\eta$ and $\xi$ invariants. For example, see equation (22) of \cite{witten1985global}, and equation (3.3) of the more recent work of \cite{cvetic2025extradimensionaletainvariantsanomalytheories}.
\end{remark}

\subsection{Spectral decomposition theorem}

The equivariant $\eta$ and $\xi$ invariants are signed, renormalized traces over the eigensections of the operator $A$ on the even degree spinor bundle, restricted to the link $Z$. On the other hand, the complex equivariant versions we define in this paper are signed renormalized traces over the harmonic sections of the Dolbeault complex. Thus the spectral data over which the traces are taken give a categorification of the equivariant $\eta$ and $\xi$ invariants as well as their complex versions. Understanding the relationship between the categorifications necessitates identifying in detail a meaningful relationship between the associated spectral data. This is the motivation behind our main result, Theorem \ref{mainthm}, which is a kind of spectral decomposition theorem. Our proof of Theorem \ref{mainthm} hinges on understanding the correspondence between supertraces over local cohomology groups and traces over the eigensections of $A$.

On the cone $X = C(Z)$ the Dolbeault-Dirac operator admits the factorization
\[
    D = \nu \circ(\nabla_{\partial_{x}}^{E} + A),
\]
where $\nu = \sqrt{2}(\beta \wedge - \iota_{\beta^{\#}})$ and $A$ is a first-order operator tangent to the link. We write $A = A_{1} + A_{2}$, with
\[
    A_{1} = -\frac{\nu}{2x}D_{1}, \mbox{ \, and \, } A_{2} = -\frac{\nu}{2x}D_{2},
\]
cf. Equation \eqref{eq:Dirac-factorization-2}. We will use the decomposition of the (even/odd) complex spinor bundles for the Dolbeault-Dirac operator,
\begin{equation}\label{dolbeault-splitting}
    F^{\pm}\big|_{Z}\;=\;G^{\pm}\oplus \beta\wedge G^{\mp},
\end{equation}
where $G^{\pm}$ are the transversal spinor bundles for $D_2$. We also define the restriction to the link at~$x=1$,
\[
    \mathcal{R}: \Omega^{0,q}(C(Z);F) \to \Omega^{0,q}(Z;F|_{Z}), \hspace{3mm} \mathcal{R}(\psi)=\psi\big|_{x=1}.
\]
Moreover, we consider the transverse harmonic subspace on the link
\begin{equation}\label{transverse-harmonic}
    \mathcal{E} := \ker(A_{2})\cap \Omega^{0,q}(Z;F|_Z),
\end{equation}
and note that on $\mathcal{E}$ we have $A = A_{1}|_{x=1} = -\tfrac{\nu}{2}D_{1}$.

Observe that the maps $\beta\wedge$ and $\iota_{\beta^\#}$ restrict to mutually inverse bundle isomorphisms
\[
    \beta\wedge:\ G^\pm \xrightarrow{\ \cong\ } \beta\wedge G^\pm,\hspace{3mm}
    \iota_{\beta^\#}:\ \beta\wedge G^\pm \xrightarrow{\ \cong\ } G^\pm,
\]
which interchange $F^{+}$ and $F^{-}$.

By unraveling the definition of the domains $\mathcal{D}_{W,B}$, and noting that we place the boundary at~$x=1$, one obtains the following. 
\begin{lemma}[Boundary conditions and the $\beta$-splitting] \label{bdy-cond-and-splitting}
    Let $P=\bar\partial_{E}$ and let $B=N,D$ denote the generalized Neumann/Dirichlet boundary conditions at $x=1$, as in Equation \eqref{dirichlet-neumann-bdy-conditions}. Take any smooth section $\psi \in \mathcal{D}_{W,B}(P) \cap \Omega^{0,q}(C(Z);F)$, so that the restriction $\mathcal{R}(\psi)$ makes sense. Then, with respect to the orthogonal splitting \eqref{dolbeault-splitting}, the boundary conditions can be equivalently formulated as follows.
    \begin{enumerate}[(i)]
        \item One has $\psi \in \mathcal{D}_{W,N}(P)$ if and only if $\sigma(P^{*})(dx)\mathcal{R}(\psi) = 0$, if and only if $\iota_{\beta^{\#}}\mathcal{R}(\psi) = 0$. Hence in this case $\mathcal{R}(\psi) \perp \beta \wedge G$ and $\mathcal{R}(\psi) \in G^{+}\oplus G^{-}$.
        
        \item One has $\psi \in \mathcal{D}_{W,D}(P)$ if and only if $\sigma(P)(dx)\mathcal{R}(\psi) = 0$, if and only if $\beta \wedge \mathcal{R}(\psi) = 0$. Hence in this case
        \[
            \mathcal{R}(\psi) \in \beta\wedge G^{-} \mbox{ if } \psi\in F^+, \hspace{3mm} 
            \mathcal{R}(\psi) \in \beta\wedge G^{+} \mbox{ if } \psi\in F^- .
        \]
    \end{enumerate}
\end{lemma}

Lemma \ref{bdy-cond-and-splitting} applies, in particular, to the harmonic sections $\psi \in \mathcal{H}(\mathcal{P}_{B}(C(Z)))$ used in the proof of Theorem \ref{mainthm}. We will also use the following lemma which may be derived from a direct calculation.
\begin{lemma}\label{A1-operator-lemma}
    On $\mathcal{E}$, the tangential operator $A_{1}$ anti-commutes with $\beta\wedge$ and with $\iota_{\beta^\#}$:
    \[
        A_{1}\circ(\beta\wedge)=-(\beta\wedge)\circ A_{1},\hspace{3mm}
        A_{1}\circ\iota_{\beta^\#}= -\iota_{\beta^\#}\circ A_{1},
    \]
    hence applying $\beta\wedge$ or $\iota_{\beta^\#}$ flips the sign of the $A_{1}$–eigenvalues.
\end{lemma}

The following refines parts (1) and (2) of Theorem~\ref{intro: mainthm} by making the correspondence between boundary conditions and spectral data explicit. The rest of Theorem ~\ref{intro: mainthm} follows from Theorem 
\begin{theorem}\label{mainthm}
Consider the twisted Dolbeault complex
\[
    \mathcal{P}_{W,B}=\bigl(L^2\Omega^{0,\bullet}(C(Z);E),\,\mathcal{D}_{W,B},\,P=\bar\partial_{E}\bigr),
    \qquad B\in\{N,D\},
\]
and let $\mathcal{H}^q_B$ denote its space of harmonic sections in degree $q$. Assume in addition the spectral Witt condition holds:
\[
    \Spec(A_{1}|_{x=1} )\cap\left(-\tfrac{1}{2},\tfrac{1}{2}\right)=\varnothing.
\]
Then $\mathcal{H}^q_N\oplus\mathcal{H}^q_D$ admits an orthonormal basis $\{\psi_j\}$ such that each $\phi_j=\mathcal{R}(\psi_j)\in\mathcal{E}$ is an $A$-eigen\-section with $A\phi_j=\mu_j\phi_j$ for some $\mu_j>0$, and we have the following characterization:

\begin{enumerate}
    \item $\psi_j\in \mathcal{H}^{\mathrm{even}}_{N}$ iff $\phi_j\in \mathcal{E}\cap G^{+}$ and $A\phi_j=\mu_j\phi_j$ with $\mu_j>0$.

    \item $\psi_j\in \mathcal{H}^{\mathrm{odd}}_{N}$ iff $\phi_j\in \mathcal{E}\cap G^{-}$ and $A (\beta \wedge \phi_j)=-\mu_j\beta \wedge \phi_j$ with $\mu_j>0$.
    
    \item $\psi_j\in \mathcal{H}^{\mathrm{even}}_{D}$ iff $\phi_j\in \mathcal{E}\cap(\beta\wedge G^{-})$, equivalently $\phi_j=\beta\wedge \Phi_j$ with $\Phi_{j}\in\mathcal{E}\cap G^{-}$, and $A\phi_j=\mu_j\phi_j$ with $\mu_j>0$.

    \item $\psi_j\in \mathcal{H}^{\mathrm{odd}}_{D}$ iff $\phi_j\in \mathcal{E}\cap(\beta\wedge G^{+})$, equivalently $\phi_j=\beta\wedge \Phi_j$ with $\Phi_j \in\mathcal{E}\cap G^{+}$, and $A\Phi_j=-\mu_j\Phi_j$ with $\mu_j>0$.
\end{enumerate}
\end{theorem}

Before proceeding to the proof, illustrate the idea behind Theorem \ref{mainthm} in the case where the link is the unit disk in the complex plane.

\begin{example} \label{Example_disc}
Consider the Dolbeault complex on $\mathbb{C}$. The even degree spinor bundle is just~$G^+$, and we consider the positive eigenvalue eigensections of 
\begin{equation}
    A_1=-\nu \circ D_1|_Z=\sqrt{2}\iota_{\beta} \circ (\beta \wedge \frac{\partial_x +ix\partial_\theta}{\sqrt{2}})|_{x=1}=i\partial_{\theta}
\end{equation}
which are $e^{ik\theta}$ for positive integer values of $k$. These admit unique harmonic extensions $z^k=x^ke^{ik\theta}$, that satisfy 
\begin{equation}
    \Big(\partial_x + i\frac{1}{x}\partial_{\theta}\Big) (x^ke^{ik\theta})=0=(\partial_x-\frac{k}{x})x^k.
\end{equation}
This shows how the positive eigenvalue eigensections of $A$ in the bundle $\mathcal{E}\cap G^{+}$ corresponds to the sections in $\mathcal{H}_N^{even}$. For those with negative integer eigenvalues of $A_1$ in $G^+$, we see that~$\beta e^{ik\theta}=(dx-ixd\theta) e^{ik\theta}$ for negative values of $k$ have positive eigenvalues $-k$ for 
\begin{equation}
    A_1=-\nu \circ D_1|_Z=-\sqrt{2}\beta \wedge \circ (-\iota_{\beta} \frac{\partial_x -ix\partial_\theta}{\sqrt{2}})|_{x=1}=-i\partial_{\theta}
\end{equation}
and that they admit unique harmonic extensions $(dx-ixd\theta) x^{(k-1)}e^{-ik\theta} =\overline{z}^{k-1}d\overline{z}$, that satisfy 
\begin{equation}
    \Big(\partial_x - i\frac{1}{x}\partial_{\theta}\Big) (x^ke^{-ik\theta})=0=(\partial_x-\frac{k}{x})x^k.
\end{equation}
when the bundle is trivialized by the section $d\overline{z}$. If trivialized with the section $\beta=d\overline{z}e^{-i\theta}$, it leads to the equation
\begin{equation}
    \Big(\partial_x+\frac{1}{x} - i\frac{1}{x}\partial_{\theta}\Big) (x^ke^{-i(k+1)\theta})=0.
\end{equation}
This shows how the negative eigenvalue eigensections in $\mathcal{E} \cap G^{+}$ corresponds to the the sections in $\mathcal{H}_D^{odd}$. In this way we see how the eigensections on the links extend to unique elements in the cohomology groups, which showcases the correspondence in Theorem \ref{mainthm}
\end{example}

\begin{proof}[Proof of Theorem \ref{mainthm}]
We have seen in Equation \eqref{transverse-harmonic} that 
\[
    \mathcal{R}: \mathcal{H}^q_N\oplus \mathcal{H}^q_D \to \mathcal{E} = \ker(A_{2})\cap \Omega^{0,q}(Z;F|_Z),
\]
for each $q$. Therefore for $\phi = \mathcal{R}(\psi)$, we have $A \phi = A_1 \phi$. Consider an orthonormal eigenbasis~$\{\phi_{\lambda,k}\}$ of $A$ in $L^2(Z;F|_Z)$, so that $A\phi_{\lambda,k}=\lambda\phi_{\lambda,k}$. By the separation analysis for Dirac operators on cones in \cite[Prop 5.33]{jayasinghe2023l2} (see Remark \ref{Remark_why_call_it_lifted_SUSY}), we may write $\psi(x,\cdot)=\sum_{\lambda,k} u_{\lambda,k}(x)\phi_{\lambda,k}(\cdot)$ and then the equation $D\psi = 0$ decouples modewise into scalar ODEs
\begin{equation}\label{kernel-equation}
    \partial_x u_{\lambda,k}(x)+\frac{\lambda}{x}u_{\lambda,k}(x)=0.
\end{equation}
Hence $u_{\lambda,k}(x)=c_{\lambda,k}x^{-\lambda}$. Equivalently, writing $\lambda=-\mu$ we obtain the form
\[
    u(x)=cx^{\mu} \, \mbox{ when } \, A\phi=-\mu\phi.
\]
The $L^2$-condition near $x=0$ for $\psi$ is simply $\int_{0}^{1}x^{2\mu}\,dx<\infty$, i.e. $\mu>-1/2$.

We now use the Dolbeault structure and the splitting \eqref{dolbeault-splitting}. On $\ker D$, the mode decomposition above shows that $\psi$ is a (possibly infinite, but convergent) linear combination of separated solutions $x^{\mu}\phi$ with $\mu>-1/2$. For $\psi$ to lie in the harmonic space of the Hilbert complex, we require simultaneously $P\psi=0$ and $P^*\psi=0$.

To construct the orthonormal basis $\{\psi_j\}$ for $\mathcal{H}_{N}^{q}\oplus \mathcal{H}_{D}^{q}$ in the statement of the theorem, we proceed mode-by-mode. First, we restrict attention to the tangential harmonic subspace $\mathcal{E}=\ker(A_2)\subset \Gamma(Z;F|_Z)$, and decompose it orthogonally as
\[
    \mathcal{E}=\bigoplus_{\lambda\in\Spec(A)} \mathcal{E}_\lambda,\hspace{3mm} 
    \mathcal{E}_\lambda:=\{\phi\in\mathcal{E}:A\phi=\lambda\phi\}.
\]
For each eigenspace $\mathcal{E}_\lambda$, choose an orthonormal basis $\{\phi_{\lambda,k}\}_k$ in $L^2(Z;F|_Z)$ adapted to the boundary/parity decomposition coming from Lemma \ref{bdy-cond-and-splitting}, i.e. with $\phi_{\lambda,k}$ lying in the appropriate summand $G^\pm$ or $\beta\wedge G^\mp$ according as we are constructing basis elements in $\mathcal{H}^{q}_{N}$ or $\mathcal{H}^{q}_{D}$ and in even or odd degree. For each such $\phi_{\lambda,k}$ with $\lambda=-\mu$ (so that the $L^2$ condition selects $\mu>-1/2$), define a separated solution on $(0,1)\times Z$ by
\[
    \psi_{\lambda,k}(x,z):= c_{\lambda,k} \, x^{\mu} \, \phi_{\lambda,k}(z),
\]
where the normalization constant $c_{\lambda,k}>0$ ensures $\|\psi_{\lambda,k}\|_{L^2}=1$. The modewise ODE analysis above implies that $D\psi_{\lambda,k}=0$ and $\mathcal{R}(\psi_{\lambda,k})=\phi_{\lambda,k}$, and the boundary condition at $x=1$ is satisfied by construction since $\mathcal{R}(\psi_{\lambda,k})$ lies in the correct summand singled out by Lemma \ref{bdy-cond-and-splitting}. Moreover, orthogonality of distinct pairs $(\lambda,k)\neq(\lambda',k')$ follows from separation of variables and orthogonality of the link eigenmodes:
\[
    \langle \psi_{\lambda,k},\psi_{\lambda',k'}\rangle_{L^2(X)}
    =\left(\int_{0}^{1} x^{\mu+\mu'}\,dx\right)\langle \phi_{\lambda,k},\phi_{\lambda',k'}\rangle_{L^2(Z)}=0 \mbox{ if } (\lambda,k)\neq(\lambda',k').
\]
Finally, any $\psi\in\mathcal{H}^q_N\oplus \mathcal{H}^q_D$ admits an expansion into $A$-eigenmodes on the link, and the kernel equation \eqref{kernel-equation} forces each coefficient to satisfy the same ODE, hence $\psi$ is an $L^2$-convergent linear combination of the $\psi_{\lambda,k}$ above. Thus we obtain an orthonormal basis by restricting to the finite subset $\{\psi_{j}\} \subseteq \{\psi_{\lambda,k}\}$ which spans $\mathcal{H}^q_N\oplus \mathcal{H}^q_D$.

The basis $\{\psi_{j}\}$ falls into one of the four categories (1)-(4) by construction. Lemma \ref{bdy-cond-and-splitting} identifies the Neumann and Dirichlet boundary conditions at $x=1$ with the ``no $\beta$-component'' and ``pure~$\beta$-component'' subspaces inside Equation \eqref{dolbeault-splitting}. This yields the four cases in the statement of the theorem:
\[
    \begin{array}{c|c|c}
        \text{parity} & B=N & B=D\\\hline
        \text{even }(F^{+}) & \mathcal{R}(\psi)\in G^{+} & \mathcal{R}(\psi)\in \beta\wedge G^{-}\\
        \text{odd }(F^{-})  & \mathcal{R}(\psi)\in G^{-} & \mathcal{R}(\psi)\in \beta\wedge G^{+}.
    \end{array}
\]
Under the spectral Witt hypothesis $\Spec(A_{Z}) \cap (-\tfrac{1}{2},\tfrac{1}{2}) = \varnothing$, the $L^{2}$-condition $\mu > -1/2$ forces~$\mu > 0$ when $A_{Z}\phi = -\mu \phi$, so every harmonic mode contributing to $\mathcal{H}_{N}^{q}\oplus\mathcal{H}_{D}^{q}$ has $\mu > 0$. Finally, in cases (3)-(4) we may take $\Phi_j=\iota_{\beta^\#}\phi_j$, which satisfies
$A\phi_j=+\mu_j\phi_j$ by Lemma \ref{A1-operator-lemma} with $\mu_{j} > 0$.
This completes the proof of the theorem.
\end{proof}

\begin{remark} \label{Remark_why_call_it_lifted_SUSY}
    It is the product nature of the conic metric on $C(Z)$ that allows for the separation of variables in \cite[Prop 5.33]{jayasinghe2023l2}, which shows that the eigensections of the Laplace-type operator on the link lift to eigensections of the Laplace-type operator on the cone. The supersymmetry of the eigensections of the cone leads to cancellations of contributions to the equivariant indices from the non-zero eigenvalue eigensections on the cone.
\end{remark}

\subsection{Equivariant index defect} \label{Subsection_equiv_ind_def}
Theorem \ref{mainthm} yields a five-part decomposition of the harmonic spinors, using the subspaces corresponding to the cases (1)-(4) plus the complement of the transverse harmonic subspace~$\mathcal{E}$~from Equation \eqref{transverse-harmonic}. We define a threefold orthogonal decomposition
\[
    L^{2}\Omega^{0,q}(Z;F^{+}|_{Z}) = \mathcal{H}_{1}^{q}\oplus \mathcal{H}_{2}^{q}\oplus \mathcal{H}_{3}^{q}
\]
where $\mathcal{H}_{1}^{q} = \mathcal{R}(\mathcal{H}_{N}^{q})$ consists of the spinors on $Z$ in cases (1) and (2) corresponding to $B = N$, and $\mathcal{H}^{q}_{2} = \mathcal{R}(\mathcal{H}^{q}_{D})$ consists of the spinors in cases (3) and (4) corresponding to $B=D$. The remaining subspace $\mathcal{H}_{3}^{q}$ is the complement of the transverse harmonic subspace $\mathcal{E}$, cf. Equation \eqref{transverse-harmonic}, within $L^{2}\Omega^{0,q}(Z;F^{+}|_{Z})$.

This in turn yields a threefold splitting of the $\xi_{T}$ function \eqref{equiv-xi-function}. Let $E_{\lambda} = \ker(A_{Z}-\lambda)$ denote the $\lambda$-eigenspace of $A_{Z}$ on $\Omega^{0,q}(Z;F^{+}|_{Z})$. Since each subspace $\mathcal{H}_{j}^{q}$ is closed and $T$ invariant, we obtain an orthogonal decomposition for each eigenspace
\[
    E_{\lambda} = \bigoplus_{j=1}^{3}(E_{\lambda}\cap \mathcal{H}_{j}^{q}).
\]
The $\eta_{T}$ function \eqref{eqn:equiv-eta-fn} then decomposes as
\begin{align*}
    & \eta_{T}(s) = \eta_{T,1}(s)+\eta_{T,2}(s)+\eta_{T,3}(s) \\
    & \mbox{with \, } \eta_{T,j}(s) = \sum_{j\not= 0}\mathrm{sign}(\lambda)|\lambda|^{-s}\Tr(T|_{E_{\lambda}\cap \mathcal{H}_{j}^{q}}).
\end{align*}
Similarly, the $T$-trace on the nullspace splits as
\[
    \Tr(T|_{\ker A_{Z}}) = h_{T,1} + h_{T,2} + h_{T,3}, \hspace{3mm} h_{T,j} = \Tr(T|_{\ker(A_{Z})\cap\mathcal{H}_{j}^{q}}).
\]
Therefore for the $\xi_{T}$ function we have the decomposition
\begin{equation}\label{threefold-partition 2}
    \xi_{T}(s) = \xi_{T,1}(s) + \xi_{T,2}(s) + \xi_{T,3}(s), \hspace{3mm} \xi_{T,j}(s) = \frac{1}{2}(\eta_{T,j}(s) + h_{T,j}).
\end{equation}
This leads to our next main result, which together with Theorem \ref{mainthm} gives Theorem \ref{intro: mainthm}.

\begin{theorem} \label{Theorem_simpler_invariant_computation}
    In the setting of Theorem \ref{mainthm}, for geometric endomorphisms $T=f^*_g$ where $f^*_g$~has an isolated fixed point at the cone point of $C(Z)$, the equivariant $\xi$ and complex equivariant~$\xi$~invariants coincide:
    \begin{equation}\label{mainthm2-eqn2}
    \widetilde{\xi}_{T}(0) = \xi_{T}(0).
    \end{equation}
    In particular $\mathcal{H}_{3}^{q}$ does not contribute to the $\xi_T$ invariant.
\end{theorem}

\begin{proof}[Proof of Theorem \ref{Theorem_simpler_invariant_computation}]
    The expressions on either side of the identity \eqref{mainthm2-eqn2} are local Lefschetz numbers. The expression on the left hand side is a \textit{symmetrized} form of the local Lefschetz formula for the cone over the link, proven when the cone has a complex structure for the Dolbeault complex in \cite{jayasinghe2023l2}. The expression on the right hand side is also the local Lefschetz number, proven in \cite{weiping1990note} when $Z$ is smooth under the assumption that $h_T=0$, but can be easily adapted to the case where $h_T \neq 0$. This proves the identity.

    Recall the definition of the function $\widetilde{\xi}_{T}(s)$ in Equation \eqref{new-equiv-xi-function}. The results in \cite[\S 5.2.5]{jayasinghe2023l2}, together with standard techniques similar to those used in \cite{weiping1990note} can be adapted to show that~$\xi_{T,1}(s)$ and~$\xi_{T,2}(s)$ admit meromorphic extensions around $s=0$, which then implies by a short calculation from the definition of $\widetilde{\xi}_{T}(s)$, that
    \[
        \widetilde{\xi}_{T}(0) = \xi_{T,1}(0) + \xi_{T,2}(0),
    \]
    i.e., $\mathcal{H}_{3}^{q}$ does not contribute to the $\widetilde{\xi}_T$ invariant. Thus the identity \eqref{mainthm2-eqn2} implies that $\xi_{T,3}(0) = 0$. Since $\mathcal{H}_{3}^{q}$ does not contribute to the $\widetilde{\xi}_T$ invariant, it does not contribute to ${\xi_T}$.
\end{proof}

We interpret the equality $\widetilde{\xi}_{T}(0) = \xi_{T}(0)$ in Theorem \ref{Theorem_simpler_invariant_computation}, and the fact that $\mathcal{H}_{3}^{q}$ does not contribute to the $\xi_T$ invariant, as exhibiting a supersymmetric cancellation at the level of the cone.

One interesting aspect of this result is that it shows that the $\xi_{T}$ invariant can be computed using supertraces over local cohomology, which can often simplify calculations, essentially because one does not need to understand the full spectrum of the Dirac operator on the link. We will exhibit this with several examples in \S \ref{section:examples}.

Given a complex structure $J_1$ on $C(Z)$, with the metric $g=dx^2+x^2S^2g_Z$ for $s=1$, we can easily extend to get a complex structure $J_S$ in general, where 
\[
    J_S(\partial_x)=Sx\alpha, \quad J_S(\frac{1}{xS}R)=-dx \quad  \text{ and } \quad J_S(V)=J_1(V)
\]
for any wedge vector field $V$ that is linearly independent from the span of $\partial_x,\frac{1}{xS}R$. The integrability of this complex structure follows from a comparison of computations in coordinates for~$\nabla^S J_S$ for general values of $S$ and for $S=1$, where the integrability holds by assumption.

Since the cohomology of the Dolbeault complex for a cone with a metric $g=dx^2+x^2Sg_Z$~with domain $\mathcal{D}_{\min}$ is isomorphic for all positive real numbers $S$ (with corresponding change in the complex structure), the equivariant index formula extends to spaces where the Dolbeault complex does not satisfy the spectral Witt condition, for the case of the minimal domain (see Remark~6.7 of \cite{jayasinghe2024holomorphic}). We next consider the case of general rolled-up extensions; essentially, the extra sections in $W$ contribute additional supertraces. We repeat Theorem \ref{intro: Theorem_general_domains_xi_invariants} for convenience.

\begin{theorem}[Theorem \ref{intro: Theorem_general_domains_xi_invariants} repeated] \label{Theorem_general_domains_xi_invariants}
    Consider the setting of Theorem \ref{Theorem_simpler_invariant_computation} but without the Spectral Witt condition. For any rolled-up extension corresponding to a choice of cohomological pre-domain $W$, the $\widetilde{\xi}_{T}$ invariants satisfy
    \begin{equation}
       \widetilde{\xi}_{T,W} =\widetilde{\xi}_{T,W}+\frac{1}{2}\Big(\Str (T|_{W^q})- \Str (T^*|_{(W^*)^q}) \Big).
    \end{equation}  
\end{theorem}

\begin{proof}
    We first observe that the cohomology groups satisfy
    \begin{equation}
        W^{q} \oplus \mathcal{H}^{q}(\mathcal{P}_{\min,N})=\mathcal{H}^{q}(\mathcal{P}_{{W},N})
    \end{equation}
    and
    \begin{equation}
        (W^*)^{q} \oplus \mathcal{H}^q(\mathcal{P}_{\max,D})=\mathcal{H}^q(\mathcal{P}_{{W^*},D})
    \end{equation}
    Since the spaces $W,W^*$ are finite dimensional, we do not need to renormalize the supertraces over them. Then the formula follows from the definition of the $\widetilde{\xi}_{T,W}$ invariant for a cohomological pre-domain $W$.
\end{proof}

\begin{example} \label{example_Lefschetz_traces_0}
We consider the cohomology of the domain described at the end of Example \ref{Example_Complex_cone_6pi}. Under the rotation $\theta \mapsto \theta + \alpha$, we get the transformation $z \mapsto z e^{i\alpha}$, so that $e^{i\theta/3} \mapsto \lambda e^{i\theta/3}$ with~$\lambda = e^{i\alpha/3}$. 
Observe that 
\begin{equation}
    \Tr T|_{\{a\tau_1+b\tau_2\}}=\frac{\langle a\lambda^{-1}\tau_1+b\lambda^{-2}\tau_2, a\tau_1+b\tau_2\rangle_{L^2\Omega(X)}}{\langle a\tau_1+b\tau_2, a\tau_1+b\tau_2\rangle_{L^2\Omega(X)}}=a^2\lambda^{-1}+b^2\lambda^{-2}
\end{equation}
where we use the fact that $\tau_1,\tau_2$ are of unit norm on the cone, and that $a^2+b^2=1$. The Lefschetz supertrace for the Dolbeault complex with domain $\mathcal{D}_{W_{a,b}}(\overline{\partial})$ having cohomology generated by~$a\tau_1+b\tau_2$ and other holomorphic modes is then
\begin{equation}
    a^2\lambda^{-1} +b^2 \lambda^{-2}+ \sum_{k=0}^\infty \lambda^k = a^2\lambda^{-1} +b^2 \lambda^{-2} + \frac{1}{1 - \lambda},
\end{equation}
ignoring convergence issues that are handled by the renormalizations. This matches the supertrace on the dual complex with domain $\mathcal{D}_{W_{a,b}}(\overline{\partial})^*$, whose local cohomology is spanned by~$bs_1-as_2$ and $\overline{z}^{k/3} \overline{dz}$ for $k \geq 0$. Here we use the fact that $\overline{dz} \mapsto\lambda^{-3}\overline{dz}$, and $\overline{z} \mapsto \lambda^{-3}\overline{z}$. Observe that 
\begin{equation}
    \Tr T|_{\{bs_1-as_2\}}=\frac{\langle b\lambda^{-1}s_1-a\lambda^{-2}s_2, bs_1-as_2\rangle_{L^2\Omega(X)}}{\langle bs_1-as_2, bs_1-as_2\rangle_{L^2\Omega(X)}}=b^2\lambda^{-1}+a^2\lambda^{-2}
\end{equation}
where we use the fact that $\tau_1,\tau_2$ are of unit norm on the cone, and that $a^2+b^2=1$. This terms contributes to the supertrace with a minus sign since the anti-holomorphic one forms are in degree $1$ in the Dolbeault complex. Together with the contributions from the supertrace over the rest of the sections of the cohomology group, we get the supertrace as
\begin{multline}
    -b^2\lambda^{-1}-a^2\lambda^{-2}-\sum_{k=-3}^{-\infty} \lambda^k = a^2\lambda^{-1}+b^2\lambda^{-2} - \sum_{k=-1}^{-\infty} \lambda^k \\=
    a^2\lambda^{-1}+b^2\lambda^{-2}- \frac{\lambda^{-1}}{1 - \lambda^{-1}}=a^2\lambda^{-1}+b^2\lambda^{-2}+
    \frac{1}{1 - \lambda}    
\end{multline}
where in the first equality we use the fact that $a^2+b^2=1$.

Thus it is easy to see that the equivariant $\xi_T$ invariant is 
\[
    \xi_{T} = a^2\lambda^{-1} +b^2 \lambda^{-2} + \frac{1}{1 - \lambda},
\]
which differs from that for the minimal domain by $a^2\lambda^{-1} +b^2 \lambda^{-2}$.
\end{example}

\begin{remark} \label{Remark_on_Chou_Zhang_work}
    In Theorem (5.23) of \cite{chou1985dirac} (c.f. Theorem 3.1 of \cite{chou1989criteria}),there are additional terms in the index defect given by dimensions of \textit{small eigenvalue} eigenspaces for the Dirac operator when the spectral Witt condition is not satisfied. 
    These small eigenvalues correspond to the limit circle cases of the singular Sturm Liouville equations when doing separation of variables (see page 3 of \cite{chou1985dirac}), and lead to more local sections in the null space of the Dirac operator in each degree.

    It is easy to check that the proof of Theorem (5.23) of \cite{chou1985dirac}, together with the equivariant version in \cite{weiping1990note}, can be extended to see understand the equivariant index defect in such cases. 
    We note that slightly different conventions in the definitions of the $\eta$ invariants lead to sign differences for results in \cite{weiping1990note,chou1985dirac}.
\end{remark}
\section{Examples and special cases}\label{section:examples}

We begin by briefly reviewing the equivariant $\eta$ invariant computations of \cite{Degeratuthesis} for the sphere (in which case it is also equal to twice the $\xi_T$ invariant since there are no harmonic spinors on the sphere and $h_T=0$), where the eigenvalues and the eigensections of the Dirac operator were studied using representation theory.

\subsection{$\eta$ invariants for the spin Dirac operator on round spheres}

We show that the new formula for the equivariant index defect given by the $\widetilde{\xi}_T$ invariant is simpler to compute than the $\xi_T$ and $\eta_T$ invariants for the spin Dirac operator on the spheres~$S^{2n-1}$. We compare this with the computations for the $\eta$ invariants using the spectrum of the Dirac operators in \cite{Degeratuthesis}, where she compares them with Molien series in the equivariant case, referring to \textit{loc. cit.} for details.

We review the proof of \cite[Theorem 4.1.1]{Degeratuthesis} which establishes a formula for the equivariant~$\eta$~invariant on $S^{2n-1}$. The spectrum of the spin Dirac operator on $S^{2n-1}=U(n)/U(n-1)$~is computed using representation theoretic methods. The computations are involved and we describe them briefly. The spectrum of the spin Dirac operator is shown to admit a three-part decomposition in Chapter 3.4 of \textit{loc. cit.}, by which the equivariant $\eta$ invariant $\eta_{T_g}(s):=\eta(g,s)$ decomposes as
\[
    \eta(g,s) = \eta_{1}(g,s) + \eta_{2}(g,s) + \eta_{3}(g,s)
\]
which corresponds exactly to our threefold splitting \eqref{threefold-partition} where $g \in \mathbb{T}^n$. The focus there is initially on the case where $g \in \mathbb{T}^n$ such that 1 is not an eigenvalue. The first contribution is
\begin{equation}
    \eta_1(g,s)=(-1)^n \sum_{a \geq 0} \Big(\frac{2n-1+2a}{2\sqrt{2}} \Big)^{-s} \det(g)^{1/2} H_a^n(g)
\end{equation}
where $H^n_a$ is the $a$-th symmetric polynomial in $n$ variables (the sum of all distinct monomials of degree $a$). Then evaluating at $s=0$ one gets
\begin{equation}
    \eta_1(g,0)=(-1)^n \det(g)^{1/2} \sum_{a \geq 0} H^n_a(g)=(-1)^n \det(g)^{1/2} \frac{1}{\det(I_n-g)}
\end{equation}
where the last equality uses the identity
\begin{equation}
    \prod_{i=1}^n \frac{1}{1-x_it}=\sum_{a \geq 0} H^n_a t^a.
\end{equation}

The computations for case 2 are similar to case 1, and it is shown that the contribution to the~$\eta$~invariant is 
\[
    \eta_{2}(g,0)=\det(g)^{-1/2}\frac{1}{\det(I_n-g^{-1})}.
\]
In the third case it is shown that
\begin{align*}
    \eta_{3}(g,s)
        &=\sum_{a,b\ge 0}\sum_{r=0}^{\,n-1}
        \Big(\frac{(-1)^{\,n+r}+2n+2(a+b)}{2\sqrt{2}}\Big)^{-s}\,
        \Tr\big(g\big|_{V_{\mu(a,r,b)}}\big)\\[4pt]
        &\qquad\qquad
        - \sum_{a,b\ge 0}\sum_{r=0}^{\,n-1}
        \Big(\frac{(-1)^{\,n+r-1}+2(n-2r-2)+2(a+b)}{2\sqrt{2}}\Big)^{-s}\,
        \Tr\big(g\big|_{V_{\mu(a,r,b)}}\big),
\end{align*}
where $V_{\mu(a,r,b)}$ is the irreducible representation of $U(n-1)$ corresponding to the highest weight
\[
    \mu(a,r,b) =\big(a+\tfrac12,\tfrac12,\dots,\tfrac12,
    -\tfrac12,\dots,-\tfrac12,-b-\tfrac12\big)
\]
with the sequence $(-1/2,\ldots,-1/2)$ having length $r$. While not identically $0$, in the limit as~$s \to 0$ the contributions cancel off to yield
\begin{equation} \label{cancellation}
    \eta_{3}(g,0) = \lim_{s \to 0} \eta_{3}(g,s) = 0.
\end{equation}
Thus the equivariant $\eta$ invariant for the odd-dimensional sphere $S^{2n-1}$ is
\begin{align*}
    \eta(g,0) &= \eta_1(g,0)+\eta_2(g,0) \\
    &=(-1)^n \det(g)^{1/2} \frac{1}{\det(I_n-g)}+\det(g)^{-1/2} \frac{1}{\det(I_n-g^{-1})} \\
    &= (-1)^n \det(g)^{1/2} \frac{2}{\det(I_n-g)} \numberthis \label{degeratu-equiv-eta}
\end{align*}
for $g \in \mathbb{T}^n$ such that 1 is not an eigenvalue. Since the Lichnerowicz identity shows that there are no harmonic spinors on the spheres, this is equal to the equivariant $\xi_T$ invariant up to a factor of $1/2$.

We emphasize that the calculations leading to Equation \eqref{degeratu-equiv-eta} required a complete analysis of the spectrum of the Dirac operator on $S^{2n-1}$. In fact, a careful examination of the supertraces over local cohomology in Equation \eqref{new-equiv-xi-function} shows that it is not necessary to understand the full spectrum of the Dirac operator in order to compute the equivariant $\xi$ invariant. Indeed, it is already seen in the cancellation \eqref{cancellation} that some of the eigenvalues give no contribution. In this sense the equivariant~$\xi$~invariant defined by Equation \eqref{equiv-xi-function} encodes superfluous information. By contrast, the invariant $\widetilde{\xi}$ \eqref{new-equiv-xi-function}, given in terms of supertraces over local cohomology, is actually easier to compute because it does not involve any excess information, and does not require a complete understanding of the spectrum.

Regarding Equation \eqref{cancellation}, we note that Theorem \ref{intro: mainthm} in the case of the cone over the odd-dimensional sphere matches with this cancellation and, in fact, Theorem \ref{intro: mainthm} implies that such a cancellation occurs more generally. As far as we know the underlying mechanism of this cancellation has not yet been explained in the literature.

To demonstrate the computational efficacy of the equivariant $\widetilde{\xi}$ invariant, we will show how to recover the formula \eqref{degeratu-equiv-eta} using the equivariant $\widetilde{\xi}_{T_g}$ invariant. The computations are much simpler, described in terms of the holomorphic sections of the spinor bundle $L=\sqrt{\mathcal{K}}$, where $\mathcal{K}$ is the canonical bundle on $\mathbb{C}^n$ and with $\mathbb{C}^{n}$ regarded as a cone over the sphere $S^{2n-1}$, and closely follows \cite[\S 7.2.6]{jayasinghe2023l2} and \cite[\S 7.3]{jayasinghe2024holomorphic} to which we refer the reader for details on the isomorphisms between the spin Dirac operator and the Dolbeault-Dirac operator acting on sections of $L$.

\begin{example}[$\eta$ invariant of odd-dimensional spheres]
\label{example_spheres}

We consider the standard torus action of on $\mb{C}^{n} = C(S^{2n-1})$, equipped with the twisted Dolbeault-Dirac complexes $\mathcal{P}_B(C(Z))$ for $B=N/D$. Fix the standard complex coordinates on $\mb C^{n}$. For a multi-index $\mathbf m=(m_{1},\dots,m_{n})\in\mb N^{n}$ write
\[
    z^{\mathbf m}:=z_{1}^{m_{1}}\cdots z_{n}^{m_{n}},\hspace{3mm}
    \bar z^{\mathbf m}:=\bar z_{1}^{m_{1}}\cdots \bar z_{n}^{m_{n}},\hspace{3mm}
    |\mathbf m|:=\sum_{j=1}^{n}m_{j}.
\]
The $\overline{\partial}$-Neumann local cohomology is spanned by the twisted monomials
\[
    \mc{H}^0\left( \mathcal{P}_{N}(C(Z))\right) = \operatorname{span}_{\mb C}\{z^{\mathbf m} \otimes \ell^{1/2}:\mathbf m\in\mb N^{n}\}
\]
where $\ell^{1/2} = \sqrt{dz_{1}\wedge dz_{2}\cdots \wedge dz_{n}}$ is a the local section that trivializes a square-root $K^{1/2}$ of the canonical line bundle. The $\overline{\partial}^*$-Neumann local cohomology is spanned by the twisted antiholomorphic $n$-forms
\[
    \mc{H}\left(P_{N}(C(Z))\right) = \operatorname{span}_{\mb C}\{\,\bar z^{\mathbf m}\,d\bar z_{1}\wedge\cdots\wedge d\bar z_{n}\otimes \ell^{1/2}: \mathbf m\in\mb N^{n}\,\}.
\]
By conjugating to an element of a maximal torus we may assume without loss of generality that~$g$ is diagonal and write
\[
    g\cdot z_{k} = \lambda_{k}z_{k}, \hspace{3mm} \lambda_{k} \in S^{1}.
\]
whence the action by $g$ on anti-holomorphic 1-forms is determined by
\[ 
    g\cdot\left(\bar z^{\mathbf m}\right) = \prod_{j=1}^{n}\lambda_{j}^{-m_{j}}\,\bar z^{\mathbf m}, \hspace{3mm} g^{*}(d\bar z_{j}) = \lambda_{j}^{-1}d\bar z_{j}.
\]
Thus, given an element
\[
    \Phi_{\mathbf m}:=\bar z_{1}^{m_{1}}\cdots\bar z_{n}^{m_{n}} d\bar z_{1}\wedge\cdots\wedge d\bar z_{n}\otimes L^{1/2} \in \mc{H}(\mathcal{P}_{D}(C(Z)))
\]
we see that the action by $g$ is given by
\begin{align*}
    g\cdot\Phi_{\mathbf m}
    = \left(\prod_{j=1}^{n}\lambda_{j}^{-m_{j}}\right)
    \left(\prod_{j=1}^{n}\lambda_{j}^{-1}\right)
    \left(\prod_{j=1}^{n}\lambda_{j}^{1/2}\right)
     \Phi_{\mathbf m}= \prod_{j=1}^{n}\lambda_{j}^{-(m_{j}+1/2)} \Phi_{\mathbf m}.
\end{align*}

For the $\overline{\partial}$-Dirichlet boundary conditions, the supertrace over local cohomology contributes
\begin{align*}
\begin{split}
   \Str\left(e^{-sD_{Z}}T_{g}\big|_{ \mathcal{H}(\mathcal{P}_{D}(C(Z)))}\right)
    &= (-1)^{n}\Tr\left(e^{-sD_{Z}}T_{g}\big|_{ \mc{H}(\mathcal{P}_{D}(C(Z)))}\right) \\
    &= (-1)^{n} \prod_{k=1}^{n} \left( \sum_{m_{k}\geq 0} \lambda_{k}^{-(m_{k}+1/2)} e^{-sm_{k}} \right) \\
    &= (-1)^{n} \prod_{k=1}^{n} \left( \lambda_{k}^{-1/2} \sum_{m_{k}\geq 0} (\lambda_{k}^{-1} e^{-s})^{m_{k}} \right).
\end{split}
\end{align*}
Summing the geometric series and letting $s\to 0$ yields
\[
    \lim_{s\to 0}\, \Str \left(e^{-sD_{Z}}T_{g}\big|_{ \mc{H}(\mathcal{P}_{N}(C(Z)))}\right)
= (-1)^{n}\prod_{k=1}^{n} \frac{\lambda_{k}^{-1/2}}{1-\lambda_{k}^{-1}}.
\]
For $\overline{\partial}$-Neumann boundary conditions, the supertrace over local cohomology contributes
\begin{align*}
    \Str \left(e^{-sD_{Z}}T_{g}\big|_{ \mc{H}(\mathcal{P}_{N}(C(Z)))}\right)
    = \prod_{k=1}^{n} \left( \sum_{m_{k}\geq 0} \lambda_{k}^{m_{k}+1/2} e^{-sm_{k}} \right)
    = \prod_{k=1}^{n} \left( \lambda_{k}^{1/2} \sum_{m_{k}\geq 0} (\lambda_{k} e^{-s})^{m_{k}} \right)
\end{align*}
and therefore
\[
    \lim_{s\to 0}\, \Str \left(e^{-sD_{Z}}T_{g}\big|_{ \mc{H}(\mathcal{P}_{D}(C(Z)))}\right) = \prod_{k=1}^{n} \frac{\lambda_{k}^{1/2}}{1-\lambda_{k}}.
\]
Now we observe that these two contributions coincide:
\[
    (-1)^{n}\prod_{k=1}^{n} \frac{\lambda_{k}^{-1/2}}{1-\lambda_{k}^{-1}} = \prod_{k=1}^{n} \frac{\lambda_{k}^{1/2}}{1-\lambda_{k}} = (-1)^{n}\det(g)^{1/2}\frac{1}{\det(I_{n}-g)}.
\]
Consequently we find that the complex equivariant $\xi$ invariant is given by
\begin{align*}
    2\widetilde{\xi}(g,0) &= \Str \left(e^{-sD_{Z}}T_{g}\big|_{ \mc{H}(\mathcal{P}_{N}(C(Z)))}\right) + \Str\left(e^{-sD_{Z}}T_{g}\big|_{ \mc{H}(\mathcal{P}_{D}(C(Z)))}\right) \\
    &= (-1)^{n}\det(g)^{1/2}\frac{2}{\det(I_{n}-g)} \\
    &= \eta(g,0)
\end{align*}
which recovers Degeratu's formula \cite[Theorem 4.1.1]{Degeratuthesis}. Note once again that $\eta(g,0) = 2\xi(g,0)$ since $h = 0$ for the sphere.
\end{example}

\begin{remark}
    In \cite{Degeratuthesis} it is shown how the case when $g$ has $1$ as an eigenvalue (including the case of $g=\mathrm{Id}$) can be derived from the earlier cases, in particular noting that the contributions from the spectrum in case 3 continues to be zero due to cancellations.
\end{remark}

\begin{remark}
    In \cite{Degeratuthesis} Chapter 4.4, Degeratu gives an equivariant $\eta$ formula for the twisted Dirac operator, corresponding to a homogeneous vector bundle $E$, which she shows admits a lift to the double cover of $U(n)$. If it admits an extension as a Hermitian bundle on the cone, then it is easy to recover her equivariant index formulas by multiplying by the formulas for the untwisted spin Dirac operator with an additional factor of $\chi_E(g)$.
\end{remark}

\subsection{Orbifolds} \label{subsection_example_orbifolds}

The case of equivariant $\eta$ invariants for orbifolds are of great interest, and are studied for quotient singularities in \cite{Degeratuthesis} as well (c.f. \cite{degeratu2003geometrical}). Our results appear to generalize to orbifolds (even without conic metrics) in a straightforward fashion, and we briefly indicate several examples in the literature to support this claim.

The zero set of the quadric $Z^2-XY=0$ in $\mathbb{C}^3$ admits an embedding into $\mathbb{CP}^3$ with one isolated conic singularity, which admits a toric action which generically has 2 smooth isolated fixed points and fixes the singular point. The equivariant index defect at the fixed singular fixed point for the spin Dirac operator was calculated in (Example 7.30 of \cite{jayasinghe2023l2}) and it is easy to see how this can be extended to the corresponding $\xi_T$ invariant. In Remark 7.32 of \cite{jayasinghe2023l2}, the correspondence between the equivariant index formula for the algebraic variety with those given by Edidin-Graham \cite{localization_algebraic_Graham_Edidin} for the same example was highlighted, and we refer the reader to those texts. We will not distinguish between the $\widetilde{\xi_T}$ and $\xi_T$ invariants in light of Theorem \ref{intro: mainthm}, and will avoid details of the renormalization (which is similar to the renormalizations in Example \ref{example_spheres} above) in what follows.

\begin{example}[Quadric cone]
    As in Example 7.28 of \cite{jayasinghe2023l2}, we consider the quadric cone $\widehat{M} = \{[W:X:Y:Z] \in \mathbb{CP}^3 \mid Z^2 = XY\}$ with $\mathbb{T}^2$-action 
    \[(\lambda, \mu) \cdot [W:X:Y:Z] = [W:\lambda^2X:\mu^2Y:\lambda\mu Z], \qquad (\lambda,\mu)\in \mathbb{T}^2.\]
    The fixed points are $a_1=[0:1:0:0]$, $a_2=[0:0:1:0]$ (smooth) and $v=[1:0:0:0]$ (the conic singularity). At $a_1$ the tangent weights are $\chi_1 = \mu/\lambda$ and $\chi_2 = 1/\lambda^2$, whence 
    \[\xi_T^{a_1}= \frac{1}{(1-\mu/\lambda)(1-1/\lambda^2)}\]
    for geometric endomorphisms $T$ corresponding to generic values of $\lambda,\mu$ for which the fixed points are isolated, where the equivariant $\xi$ invariant at $a_1$ is the index defect of the smooth cone over the sphere corresponding to the unit ball on the tangent space at $a_1$.
    Similarly, $\xi_T^{a_2}= \frac{1}{(1-\lambda/\mu)(1-1/\mu^2)}$. Near $v$, using coordinates $(x,y,z)=(X/W,Y/W,Z/W)$, the action is $(x,y,z)\mapsto(\lambda^2x,\mu^2y,\lambda\mu z)$, and the local holomorphic functions are generated by~$x^ay^b$ and $z\,x^ay^b$. Thus 
    \begin{align*}
        \xi_T^v
        &= \sum_{a,b \geq 0} \lambda^{2a}\mu^{2b} + \sum_{a,b \geq 0} \lambda^{2a+1}\mu^{2b+1} \\
        &= \frac{1}{(1-\lambda^2)(1-\mu^2)} + \frac{\lambda\mu}{(1-\lambda^2)(1-\mu^2)} = \frac{1+\lambda\mu}{(1-\lambda^2)(1-\mu^2)}.
    \end{align*}
    A direct computation shows $\xi_T(\widehat{M}) = \xi_T^{a_1}+\xi_T^{a_2}+\xi_T^{v}=1$, giving the global equivariant index of~$\widehat M$ as the sum of its local $\xi_T$ invariants.
\end{example}

\begin{example}[Quadric cone, spin $\xi$ invariants]
    In the spin Dirac framework \cite[Example~7.30]{jayasinghe2023l2}, the same torus action induces characters on the spinor bundle. At the smooth fixed points $a_1$ and $a_2$, the spin contributions coincide with the Atiyah-Bott formula, now interpreted for the spin Dirac operator. At the singular vertex $v$, one must work with the tangent cone, which in this case is a cone over a $\mathbb{Z}_2$-quotient of $S^3$. The local spinor fields are generated by monomials of the form
    \[
        \{\,x^a y^b, z x^a y^b : a,b \geq 0\,\},
    \]
    exactly as in the Dolbeault case, but up to a product of the local trivializing section of the square-root of the canonical bundle given by~$(dx\wedge dy/(2z))^{1/2}$, the square-root of the Poincar\'e residue (see Theorem 2.2 of \cite{hitchin1974harmonic}, Section 7.3 of \cite{jayasinghe2024holomorphic}). Taking the supertrace of the~$\mathbb{T}^2$-action on this local cohomology produces
    \[
        \xi_T^v = \frac{(1+\lambda\mu)\sqrt{\lambda\mu}}{(1-\lambda^2)(1-\mu^2)},
    \]
    which agrees with the Dolbeault computation. Thus, the local $\xi_T$ invariant at the singular vertex naturally appears as the spin Lefschetz contribution of the cone. Summing over all three fixed points again yields
    \begin{align*}
        &\xi_T^{a_1} + \xi_T^{a_2} + \xi_T^v
        = \frac{\sqrt{(\mu/\lambda)(1/\lambda^2)}}{(1-\mu/\lambda)(1-1/\lambda^2)} + \frac{\sqrt{(\lambda/\mu)(1/\mu^2)}}{(1-\lambda/\mu)(1-1/\mu^2)} + \frac{(1+\lambda\mu)\sqrt{\lambda\mu}}{(1-\lambda^2)(1-\mu^2)} = 0.
    \end{align*}
\end{example}

The case of an equivariant index defect for a teardrop orbifold was studied in \cite[\S 3]{meinrenken1998symplectic} using the cohomological equivariant index formulas of Vergne \cite{vergne1994quantification}. At the singular point of~$\mathbb{C}/\mathbb{Z}_k$~with coordinate $w^k$ equipped with the action $z \mapsto e^{i\phi} z$ where $z=w^{-k}$, the local equivariant index is shown to be 
\begin{equation}
    \frac{1}{k} \sum_{l=0}^{k-1} \frac{1}{1-c^le^{-\frac{i}{k}\phi}}, \quad c=e^{-2\pi/k}
\end{equation}
and it is only after simplifying the given expression to $\frac{1}{1-e^{i\phi}}$, that the power series expansion matches the supertraces over the local cohomology groups, as was shown in \cite[\S 7.3.2]{jayasinghe2023l2}. This shows that the formulas for the equivariant indices in \cite{vergne1994quantification} do not easily admit categorifications by spectral data.
\section{Context and conclusions}\label{section_conclude}

Historically, the study of $\eta$ invariant and index defects was motivated by questions posed by Hirzebruch concerning the signature defect arising from cusp ends on modular surfaces, that is, the failure of the Hirzebruch signature theorem in their presence. The work included deriving explicit number theoretic formulae for it; see Section 3 of \cite{hirzebruch1973hilbert}. In particular, the case of quotient singularities for $\mathbb{C}^n/G$ where $G$ is the group of $p$-th roots of unity was proven by resolving singularities. Equivariant versions of these invariants were introduced in \cite{donnelly1978eta} to describe equivariant index defects on manifolds with boundary corresponding to group actions on such spaces.

For singular spaces, Cheeger showed that the index defect for a Dirac operator at a conical singularity was described by an $\eta$ invariant on the link $Z$ \cite{cheeger1983spectral} (c.f. \cite{chou1985dirac,chou1989criteria}), which includes many cases of orbifolds and algebraic varieties. The corresponding equivariant version was studied by Zhang \cite{weiping1990note}. The work of Baum-Fulton-MacPherson-Quart \cite{baum1979lefschetz,Baumformula81,baumfultonmacKtheoryRiemannRoch,baumfultonmacpheresonRiemannRoch} explained in the introduction has been extended to other formulations using algebraic geometric tools, for instance in \cite{equivariant_intersection_edidin,localization_algebraic_Graham_Edidin,MaximJorgCharacteristicsingulartoric2015}. Our work in this article is a step towards bridging these descriptions.

\subsection{Supersymmetric counting and different categorifications}

Supersymmetric (SUSY) quantum field theory (QFT) is widely used to understand qualitative features of QFT's. The main reason is that one can do explicit computations of path integrals and other physically relevant observables. Roughly, asymptotic expansions of such integrals will have cancellations in all higher order terms except for the first term when there is supersymmetry, akin to how the Duistermaat-Heckman theorem states that all higher order terms in the semi-classical Fourier transform of the moment map of a symplectic space vanish. It is well known that both of these results correspond to versions of Atiyah-Bott localization, and are corollaries of appropriate equivariant index theorems. Thus the character formulas at fixed points, boundaries and singularities corresponding to equivariant index theorems for various twisted Dirac operators correspond to supersymmetric state counting. We refer the reader to \cite{witten1981dynamical,witten1983fermion} and \cite{pestun2017localization} for a more expansive overview.

Moreover these integer coefficients of the character formulas appearing in physical problems, are sometimes interpreted as the virtual dimensions of relevant moduli spaces of equivariant instantons (see for instance \cite{festuccia2020twisting,pestun2012localization}). In this light, the different categorifications of the RR numbers lead to different classes of objects sharing the virtual dimensions of the moduli spaces, namely the eigenspaces of the Dirac operator on the link, vs the local cohomology groups.

In fact in the work in \cite{festuccia2020transversally,festuccia2020twisting}, the authors use expansions of the equivariant indices at smooth fixed points in terms of holomorphic and anti-holomorphic forms, as opposed to equivariant $\eta$ invariant expansions which would lead to additional terms before cancellations. While this is natural to do in the smooth case instead of treating the tangent plane at a fixed point $\mathbb{C}^n$ as a cone over a sphere, in the singular setting, our exposition shows how to reconcile the differences between the equivariant $\eta$ invariant approaches and the local cohomology based approach.

Equivariant and non-equivariant $\eta$ invariants and $\xi$ invariants also show up when studying determinant line bundles and understanding gravitational anomalies \cite{witten1985global}, and sometimes with different renormalization schemes (see Remark \ref{Remark_different_renormalizations}). In such instances, our work explains symmetries of the objects that categorify these invariants when there is a complex structure.

The equivariant RR defects at singularities (and boundaries) admit several interesting categorifications. The RR number of a space with a complex structure and a Hermitian bundle is the (supersymmetric) count of the global sections of the Hermitian bundle, expressed in terms of cup products of characteristic classes related to the bundle. This count is categorified to the level of cohomology and further to the level of K-theory by the Hirzebruch-RR and Grothendieck-RR theorems.

The equivariant RR number admits a different categorification to a subcomplex of the Dolbeault complex called the holomorphic Witten instanton complex, first proposed by Witten (see \cite{witten1984holomorphic,mathai1997equivariant}, while the $\mathbb{Z}$-graded complex fails to hold for certain non-K\"ahler spaces (see \cite{wu1996equivariant}), the local equivariant RR number is still a supertrace over the local cohomology groups, thus providing an infinite dimensional analog of the index bundle in ($\mathbb{Z}_2$-graded) K-theory, near fixed points. There are recent studies of eta forms in differential K-theory \cite{liu2021equivariant} which also seem to admit simplifications in this way. The $\xi_T$ invariant of an operator $A$ is naturally categorified by the eigenspaces of the Dirac operator, and our results show that only certain eigenspaces contribute meaningfully. This informs constructions of equivariant index bundles at the fixed points.

We do not study orbifold equivariant index theorems such as in \cite{meinrenken1998symplectic,vergne1994quantification} in this article (confining ourselves to the discussion at the end of  \S\ref{subsection_example_orbifolds}), where the approach hinges on associating linear algebraic formulas on a smooth local cover near the singular fixed points to the fixed point on the singular space. Even in the smooth setting, we emphasize the fact that the categorifications of local contributions to the equivariant index stem from other interpretations of the local Lefschetz numbers.

\subsection{Conjectural extensions}

An obvious question here is the following.
\begin{question} 
    In the almost complex setting, can the $\eta_T$ invariant for geometric endomorphisms $T$ with non-isolated fixed points (including $T=\mathrm{Id}$, in which case it is the $\eta$ invariant) be expressed as a renormalized supertrace over the local cohomology groups? In other words, does Theorem~\ref{Theorem_simpler_invariant_computation} extend to more general $T$, including $T=\mathrm{Id}$?
\end{question}

We present a heuristic which suggest that the answer to the first question is yes, explaining the difficulties in the proof. 
The problem here is that the $\eta$ invariant is studied via the renormalization corresponding to the Dirac operator $A$ on the link, which yields the equivariant eta function~$\eta_T(s)$~as an intermediary, and as the examples in Section \ref{section:examples} and the work in \cite{Degeratuthesis} shows, $\eta_T(s)$ only admits the cancellations at $s=0$. 
We see two possible ways to navigate this problem.
\begin{enumerate}
    \item A different renormalization of the $\eta$ invariant which incorporates the information from the supersymmetry on the cone over the link or another appropriate associated geometric object where the supersymmetric cancellations are clearer.
    
    \item Studying the eigenspaces of the operator directly, and understanding symmetries between the positive and negative eigenspaces of the Dirac operator $A$ directly.
\end{enumerate}

The fact that $\xi_{T,3}(s)$ in Theorem \ref{intro: mainthm} does not vanish for general $s$ shows that there is no SUSY cancellation (in the conventional sense) happening at the level of the alternating trace over the eigensections of the operator $A$, even though renormalized supersymmetric cancellations in \cite[\S 5.2.5]{jayasinghe2023l2} are why we see that the partition $\mathcal{H}^q_3$ does not contribute to the equivariant index defect. Our contention is that the renormalizations used in the definition of the $\eta$ invariant, using the spectral information on the link, are not well suited for this. We refer to \cite{jayasinghe2023l2} (see Lemma 3.13) for a description of SUSY. This shows that for a given Dirac type operator $D$, if there is an eigensection $\phi$ with eigenvalue $\mu$, then the section $D\phi$ is also an eigensection of $D$~with the same eigenvalue $\mu$. It is this cancellation that leads to the vanishing of $\xi_{T,3}(0)$, but is seen only at the level of the eigenvalues of the operator $D$ on the cone, not the induced operator $A$ on the link, which is why $\xi_{T,3}(s)$ does not vanish identically.

In \cite{capoferri2025hodgetheoremcurlasymmetric,capoferri2025microlocalpathwayspectralasymmetry,capoferri2025spectralasymmetrypseudodifferentialprojections} the authors study delocalized versions of the $\eta$ invariant in terms of pseudodifferential operators, with projectors to the negative and positive eigenspaces of the Dirac operator on the link playing a key role. This aligns with our approach to understand the spectral asymmetry intuitively as a difference between numbers of positive and negative eigenvalues, incorporating cancellations in our case. We refer the reader to \cite{goette2012computations,goette2009} for other studies on delocalized equivariant $\eta$ invariants.

The connections between the spectral flow and the $\eta$ invariant are also of interest in physics (see sections 12-14 of \cite{nakahara2018geometry}). These connections appear when studying relative index theorems, which have also been used to develop RR results (see for instance \cite[\S 6.3]{melroseAPS} for the case of general divisors on surfaces).

The results seem to easily extend to more general singular metrics, for instance as studied in \cite{jayasinghe2025SUSY}, as well as wedge metrics on cones $C(Z)$ where $Z$ is a stratified pseudomanifold, for suitable choices of domains. In fact, the work in \cite{jayasinghe2023l2} for spin-c Dirac operators in the almost complex setting hints at extensions to that setting as well.

Moreover these symmetries should lead to simplifications of eta forms, Bismut-Cheeger~$\mathcal{J}$~forms and other objects arising in formulations of index defects at singularities \cite{Albin_2017_index,wang2011noteequivariantetaforms} when there are complex structures, all needing analytic techniques developed to incorporate these symmetries. Explicit computations show that these are connected to spectral theoretic descriptions of equivariant RR numbers, similar to those in \cite{wu1998equivariant} where such descriptions are given near non-isolated fixed point sets. We will investigate this further in upcoming work.

\subsection{Algebraic descriptions of $\eta$ invariants}

The studies of \cite[\S 5]{Degeratuthesis} of the correspondence between equivariant $\eta$ invariants for spin Dirac operators and Molien series allows one to more easily compute $\xi_T$ invariants using algebraic data.

In \cite{baum1979lefschetz,Baumformula81}, the authors show that the Lefschetz-Riemann-Roch (LRR) number (in singular cohomology) for a complete intersection can be expressed as a regularized supertrace over the local ring, the analog of the local cohomology group for $P$. In the case of normal algebraic varieties, with vanishing higher local $L^2$ cohomology, the Hilbert space completion of the local ring can be identified with the local cohomology group. The formulas in \textit{loc. cit.} seem to admit extensions to the non-equivariant case in the singular setting stated as discussed above. This would be an \textbf{algebraic $\eta$ invariant}, which could be interesting to try to extend to non-normal algebraic varieties to gain a better understanding of the index defects of the RR theorems in \cite{baumfultonmacKtheoryRiemannRoch,baumfultonmacpheresonRiemannRoch}. Equivariant index theory in complex geometry which builds on \textit{loc. cit.} (see for instance \cite{equivariant_intersection_edidin,localization_algebraic_Graham_Edidin,cappell2023equivariant}) also can be understood in this way.

In the complex case, since the equivariant signature is given by the $\chi_{+1}$ invariant (generalizing Hirzebruch's formula to the singular $L^2$ setting, see for instance in \cite{jayasinghe2023l2}), where the supertraces over different local cohomology groups for values $p$ in the Hodge decomposition are used to denote the equivariant signature. These supertraces correspond to the equivariant $\eta$ invariants, and conjecturally for (non-equivariant) $\eta$ invariant as well in the (almost) complex setting.

%------------------------------
% appendix and references
%------------------------------
\bibliographystyle{alpha}
\bibliography{references}
\end{document}